\documentclass[reqno]{amsart}

\pdfoutput = 1

\usepackage[english]{babel}
\usepackage{amsmath,amssymb}
\usepackage{amsfonts}
\usepackage{amsthm}
\usepackage{array,dcolumn}
\usepackage{graphicx}
\usepackage{algorithm}
\usepackage{algpseudocode}
\usepackage{todonotes}

\usepackage{tikz}
\usetikzlibrary{arrows,calc,decorations.pathmorphing}
\usepackage{pgfplots}

\usepackage[T1]{fontenc}
\usepackage[utf8]{inputenc}
\usepackage[activate={true,nocompatibility},spacing,kerning]{microtype}
\usepackage{bm}
\usepackage[sc,osf]{mathpazo}
\graphicspath{{plots/}}

\renewcommand{\leq}{\leqslant}

\newcommand{\C}{\mathbb{C}}

\newcommand{\I}{\mathrm{I}}

\let\join\relax
\DeclareMathOperator*{\argmin}{argmin}
\DeclareMathOperator*{\argmax}{argmax}
\DeclareMathOperator*{\res}{res}
\DeclareMathOperator*{\GL}{GL}
\DeclareMathOperator{\join}{\dot{+}}
\DeclareMathOperator{\ind}{ind}
\DeclareMathOperator{\spa}{sp}
\DeclareMathOperator{\interior}{int}

\newtheorem{lemma}{Lemma}

\begin{document}

\renewcommand{\figurename}{\small {\sc Figure\@}}
\renewcommand{\tablename}{\small {\sc Table\@}}

\title[Automatic Deformation of Riemann--Hilbert Problems]{Automatic Deformation of Riemann--Hilbert Problems with Applications to the Painlevé II Transcendents}

\author{Georg Wechslberger}
\address{Zentrum Mathematik -- M3, Technische Universität München,
         80290~München, Germany}
\email{wechslbe@ma.tum.de}
\author{Folkmar Bornemann}
\address{Zentrum Mathematik -- M3, Technische Universität München,
         80290~München, Germany}
\email{bornemann@tum.de}
\date{\today}

%\subjclass[2010]{65E05, 65D25; 68R10, 05C38}

\begin{abstract}

The stability and convergence rate of Olver's collocation method for the numerical
solution of Riemann--Hilbert problems (RHPs) is known to depend very sensitively 
on the particular choice of contours used as data of the RHP. By {\em manually}
performing contour deformations that proved to be successful in the asymptotic
analysis of RHPs, such as the {method of nonlinear steepest descent}, 
the numerical method can basically be preconditioned, making it asymptotically stable. 
In this paper, however, we will show that most of these preconditioning deformations, including lensing, can be addressed
in an \emph{automatic}, completely algorithmic fashion that would turn the numerical method into a black-box solver.
To this end, the preconditioning of RHPs is recast as a discrete, graph-based optimization problem: the deformed
contours are obtained as a system of shortest paths within a planar graph weighted by the relative strength of the jump
matrices. The algorithm is illustrated for the RHP representing the Painlevé II transcendents.

%	Calculating a numerical solution of a Riemann-Hilbert problem has recently been enabled by the introduction of a solver for these problems by \citeasnoun{Olver:2011:NSR:1967343.1967345}. Nevertheless it is often not possible to directly calculate the solution of the standard form of a Riemann-Hilbert problem, due to a high condition number of this form. As has been shown by \citeasnoun{1205.5604} it is possible to circumvent this problem by deforming the Riemann-Hilbert problem using the method of nonlinear steepest descent. Currently these deformations have to be derived by hand and therefore cause quite a lot of work that has to be done before a solution can be calculated. In this paper we propose an algorithm that can calculate these deformations in not too complicated situations. Similar to the approach used by \citeasnoun{1107.0498v2}, we turn the problem of finding a suitable deformation into an optimization problem in graph theory and solve it.

\end{abstract}

\maketitle

\section{Introduction}\label{sec:intro}

Remarkably many integrable problems in mathematics, mathematical physics, and applied mathematics can be cast as Riemann--Hilbert problems (RHPs): classical orthogonal polynomials and special functions, Painlevé transcendents, nonlinear PDEs related to the inverse scattering transform, and distributions in random matrix theory as well as in random combinatorial problems, to name just a few~\cite{MR2011605}. A fruitful point of view is that RHPs generalize the representation of classical special functions
by contour integrals; like these they have extensively been used in establishing deep asymptotic results such as connection
formulae for the Painlevé transcendents \cite{fokas:2006:ptr}. Here, a fundamental tool is the \emph{method of nonlinear steepest descent}, which was introduced by Deift and Zhou \cite{MR1207209} to the asymptotic analysis of oscillatory RHPs.

Only quite recently, starting with a novel direct spectral collocation method of Olver \cite{Olver:2011:NSR:1967343.1967345}, RHPs have become the subject of study in numerical analysis. It turned out that the stability of this numerical method
and its approximation properties strongly depends on matching most of the steps from the asymptotic analysis of the problem at hand: ``One can expect that whenever the method of nonlinear steepest descent produces an asymptotic formula, the numerical method can be made asymptotically stable'' \cite[p.~2]{1205.5604}. This way, the method is kind of {\em hybrid}: only after manually performing a series of expert analytic steps to deform the given RHP into another, equivalent, one, the thus ``preconditioned'' RHP is taken as input to the numerical algorithm. Hence, the use of this numerical method has been limited so far to an audience that would have competent operational access to such a kind of expert knowledge. 

In this paper we will give a \emph{proof of concept} that most, if not all, of these deformations (at least if they were meant to
stabilize the numerical method) can be 
addressed with an automatic, completely algorithmic approach that would turn the numerical method into a black-box solver for the user.\footnote{To begin with, in this paper we study plain contour deformations and lensing; other important concepts of deformations, such as $g$-functions, will be subject of subsequent refinements of our work.} Following 
Bornemann and Wechslberger \cite{1107.0498v2}, who dealt with similar problems for contour integrals, we will recast the preconditioning
of RHPs as a discrete, graph-based optimization problem: the desired deformation corresponds to a system of shortest paths within a weighted planar graph.\footnote{A {\em walk}  in a graph is a sequence of adjacent vertices, a {\em path} is a simple (non self-intersecting) walk.} 
Though we are not able, at this stage of our study, to determine the precise complexity class of this particular discrete optimization problem
(NP-hard, polynomial, etc.) or to prove that our polynomial greedy algorithm would approximate 
 the optimum within a certain range (which would be sufficient for the purpose of preconditioning), we will demonstrate for the example of the Painlevé II transcendents that, first, it will improve the stability of the input RHP significantly by several orders of magnitude and that, second, the resulting deformations of the RHP closely match what people have obtained by applying the method of nonlinear steepest descent. 
 
\subsection*{Riemann--Hilbert problems} To fix the notation, we 
consider RHPs for a given oriented contour $\Gamma$, which is
a finite union of simple smooth curves $\Gamma_j$ ($j=1,\ldots,k$) in $\C$. 
By removing the finitely many points of self-intersection from $\Gamma$ we obtain $\Gamma^0\subset \Gamma$. 
Given a matrix-valued jump function $G:\Gamma^0\to \GL(m,\C)$, the RHP 
determines a holomorphic function $\Phi : \C \setminus \Gamma \to \GL(m,\C)$ satisfying\footnote{The second condition is meant to imply that $\Phi$ has a holomorphic continuation at $\infty$.}
\[
%	\Phi &\text{ is analytic in } \C \backslash \Gamma \\
	\Phi^+(z) = \Phi^-(z) G(z) \quad (z \in \Gamma^0), \qquad
	\Phi(\infty) = \I.
\]
Here, $\Phi^\pm(z)$ denotes the non-tangential limit of $\Phi(z')$ as $z' \to z$ from the positive (negative) side of the contour. 
Existence and uniqueness of a solution $\Phi$ can be shown under some appropriate
smoothness and decay assumptions on the jump function $G$  \cite{MR1677884}. To simplify the discussion of
contour deformations, we assume that there are entire functions $G_j: \C\to \GL(m,\C)$ that continue
the jump data $G$ given on the part $\Gamma_j$ of the contour $\Gamma$:
\[
G|_{\Gamma^0\cap\Gamma_j} = G_j|_{\Gamma^0\cap\Gamma_j}.
\]
We consider the pairs $(\Gamma_j,G_j)$ ($j=1,\ldots, k$) as the data of the RHP. Most often one is not interested
in the full solution $\Phi(z)$ of the RHP but only on some derived quantities at $\infty$, e.g., the
residue
\[
\res_{z=\infty} \Phi(z) = \lim_{z\to\infty} z (I - \Phi(z)).
\]

\subsubsection*{Example: Painlevé II} Throughout this paper we will illustrate our ideas for the RHP representing
the  Painlevé II equation
\[
u_{xx} = x u + 2u^3.
\]
The general solution $u(x)=u(x;s_1,s_2)$ of this second-order ODE in the complex domain will depend on two independent complex parameters $s_1$ and $s_2$,\footnote{In the singular case $s_1=s_2=\pm i$, there is 
a one-parameter family of solutions depending on $s_3$.} which are fixed in the following setup of the RHP: with the six rays (see Fig.~\ref{fig:rays})
\[
\Gamma_j=\{s e^{i \pi(2j-1)/6}: s \ge 0\}\qquad (j=1,\ldots,6),
\]
parameters $s_j$ ($j=1,\ldots,6)$ interrelated by
\[
s_1 -s_2+s_3+s_1s_2s_3 =0, \quad s_4=-s_1,\quad s_5=-s_2, \quad s_6=-s_3,
\] 
and the jump matrices 
\[
G_j(z) = \begin{cases}
\begin{pmatrix}
1 & s_j e^{-\theta(z)} \\
0 & 1
\end{pmatrix} & \text{$j$ even}, \\*[6mm]
\begin{pmatrix}
1 & 0 \\
s_j e^{+\theta(z)} & 1
\end{pmatrix} & \text{$j$ odd},
\end{cases}
\]
with the phase function 
\[
\theta(z) = \frac{8i}{3}z^3+2i x z
\]
the solution $\Phi$ of the RHP yields 
\[
u(x;s_1,s_2) = -2 \res_{z=\infty} \Phi_{1,2}(z) = 2\lim_{z\to\infty} z\,\Phi_{1,2}(z).
\]
Note that the independent variable $x$ of the ODE enters the RHP as a parameter of the phase function $\theta$: the RHP
(with independent variable $z$) amounts thus for a \emph{pointwise} evaluation of the Painlevé transcendent $u(x;s_1,s_2)$.

\begin{figure}[tbp]
\begin{minipage}{0.375\textwidth}
\begin{center}
	\includegraphics[width=\textwidth]{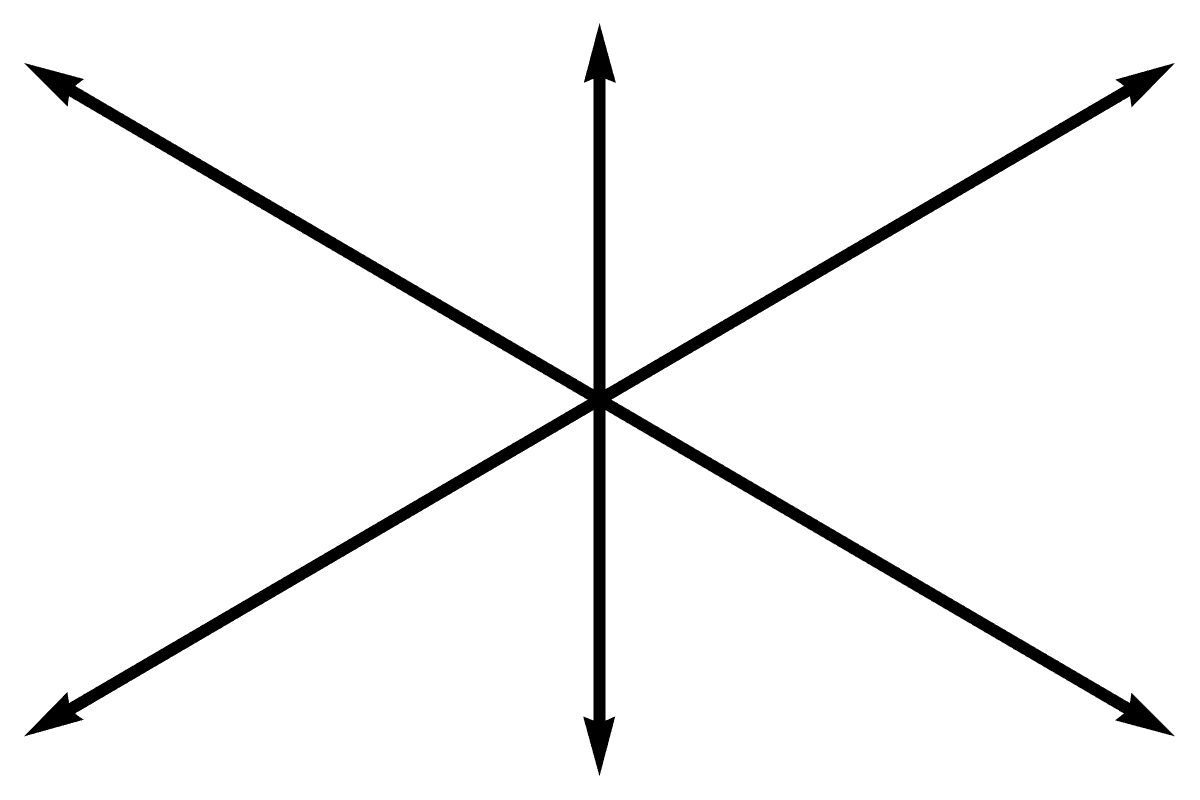}
\end{center}
\end{minipage}
\caption{The six rays $\Gamma_j$ of the RHP representing Painlevé II.}
\label{fig:rays}
\end{figure}

\subsection*{The numerical method of Olver} Olver \cite{Olver:2011:NSR:1967343.1967345} constructed his
spectral collocation method by recasting RHPs as a particular kind of singular integral equation. Upon
writing
\[
\Phi(z) = I + C_\Gamma U(z)
\]
with the Cauchy transform of a matrix-valued function $U : \Gamma \to \GL(m,\C)$, namely
\[
C_\Gamma U(z) = \frac{1}{2\pi i} \int_\Gamma \frac{U(z)}{\zeta - z}\,d\zeta,
\]
the RHP becomes the linear operator equation
\begin{equation}\label{eq:sie}
A U (z) = U(z) - C_\Gamma^- U(z) \cdot (G(z)-I) = G(z)-I. 
\end{equation}
Here, $C_\Gamma^\pm U(z)$ denotes the non-tangential limit of $C_\Gamma U(z')$ as $z' \to z$ from the positive (negative) side of the contour; there is the operator identity $C_\Gamma^+ - C_\Gamma^- = I$. The residue at $\infty$ becomes
simply the integral
\[
\res_{z=\infty} \Phi(z) = \frac{1}{2\pi i}\int_{\Gamma} U(\zeta)\,d\zeta.
\]
Without going into details, in this paper it suffices to note that the $n$-point numerical approximation of
(\ref{eq:sie}) yields a finite-dimensional linear system
\[
A_n U_n = b_n
\]
where the $j$th component of the solution $U_n$ is a matrix that approximates $U(z_j)$ at the collocation point $z_j \in \Gamma$ ($j=1,\ldots,n$). The stability of the method is essentially described by the condition number 
\[
\kappa_n = \kappa(A_n) = \|A_n^{-1}\|\cdot \|A_n\| 
\]
of this linear system: altogether, one would typically suffer a loss of $\log_{10}\kappa_n$ significant digits.
Under an additional assumption, which can be checked {\em a posteriori} within the numerical
method itself, Olver and Trogdon \cite[Assumpt.~6.1 and Lemma~6.1]{1205.5604} proved a bound of the form\footnote{They employ
the estimate $\|A\| \le \sqrt 2(1+ \|G-I\|_{L^\infty(\Gamma)}\|C_\Gamma^-\|)$ in their statements.}
\[
\kappa_n = O(\kappa(A)) 
\]
in terms of the condition number $\kappa(A) = \|A^{-1}\|\cdot \|A\|$ of the continuous operator $A$, 
with constants that are midly growing in the number of collocation points $n$. Here, the operator norm of $A$
is obtained by acting on $L^2(\Gamma)$. Extending $U_n$ to all of $\Gamma$ by interpolation, Olver and
Trogdon \cite[Eq.~(6.1)]{1205.5604} also state an error estimate of the form
\[
\|U-U_n\|_{L^2(\Gamma)} \leq c \kappa(A) n^{2+\beta-k} \|U\|_{H^k(\Gamma)}
\]
with some $\beta > 0$. Since, for jump matrices $G$ that are piecewise restrictions of
entire functions, $k$ can be chosen arbitrarily large, one gets spectral accuracy. 

\begin{figure}[tbp]
	\begin{center}
		\begin{tikzpicture}
			\begin{semilogyaxis}[xlabel={\small $x$},
				ylabel={\small condition number},
				grid=major,
				no markers,
				legend pos=north east,
				ymin=1,
				xmin=-30,
				xmax=-10
				]
				\addplot[very thick,blue] table[x index=0,y index=1] {plots/PII-conds.dat};
				\addplot[very thick,red] table[x index=0,y index=2] {plots/PII-conds.dat};
				\addplot[very thick,black,dashed] table[x index=0,y index=1] {plots/PII-conds-eps.dat};
				\legend{{\small original contour},{\small deformed contour}}
			\end{semilogyaxis}
		\end{tikzpicture}
	\end{center}
	\caption{Comparison of the condition number of the original contour and of a deformed contour optimized by the greedy algorithm of 
	§\ref{sec:algorithm} for the Painleve II RHP with $(s_1,s_2) = (1,2)$. The condition number of the deformed contour is roughly constant for all values of $x$ while the condition number of the original contour grows exponentially fast for decreasing values of $x$. Note that condition numbers larger
	than $10^{16}$ (dashed line) obstruct the computation of even a single accurate digit in machine arithmetic, indicating severe numerical instability.}\label{fig:cond}
\end{figure}
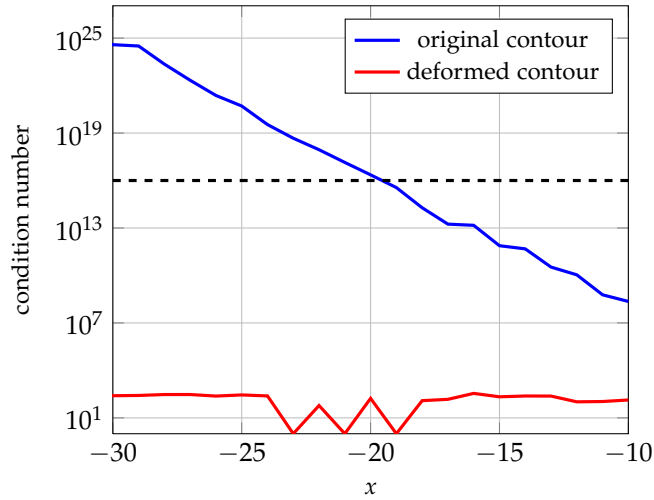

\subsection*{Preconditioning of RHPs} Thus, stability and accuracy of the numerical method depend on $\kappa(A)$, which blows
up in many problems of interest. For instance, the undeformed version of the RHP for Painlevé II with the contours in Fig.~\ref{fig:rays} has $\kappa_n \approx 2.2\cdot 10^8$ for $s_1=1$ and $s_2=2$ and $x=-10$ (see also Fig.~\ref{fig:cond} vor varying $x$). 

Now, it is important to understand that $\kappa(A)$ is the condition number of the RHP for the 
restricted data $(\Gamma,G)$ but not for the jump data $G_j$ ($j=1,\ldots,k$) 
which are obtained from analytic continuation. If the 
continued data are \emph{explicitly} given, and are not themselves part of the computational problem,\footnote{Analytic continuation corresponds to solving a Cauchy problem for
the elliptic Cauchy--Riemann differential equations; it is, therefore, an ill-posed problem.} it should be possible
to deform the RHP to an equivalent one with data $(\tilde \Gamma, \tilde G)$ and $$\kappa(\tilde A) \ll \kappa(A).$$
We call such a deformation a {\em preconditioning} one.
In fact, Olver and Trogdon \cite{1205.5604} argued that preconditioning \emph{is} possible whenever
the method of nonlinear steepest descent produces an asymptotic formula; Fig.~\ref{fig:olver} shows a typical
sequence of such manually constructed preconditioning deformations for the Painlevé II RHP. 

Though it seems to be difficult to extract a single governing principle for all the ingenious deformations
that are used in the asymptotic analysis of RHPs, we base our algorithmic approach on the following simple
observation: if there are no jumps in the RHP, that is if $G\equiv I$, we have $A=I$ and therefore $\kappa(A)=1$.
By continuity, $G \to I$ in some sufficiently strong norm would certainly imply $\kappa(A)\to 1$, such that a
reasonably small $\|G-I\|$ will probably yield a moderately sized condition number $\kappa(A)$. 
We conjecture that such an estimate can be cast in the form
\[
\kappa(A) \leq \phi(\|G-I\|_{W^{s,p}(\Gamma)})
\]
for some Sobolev $W^{s,p}$-norm and some monotone function $\phi$ that is independent of $(\Gamma,G)$. A good
preconditioning strategy would then be to make $\|G-I\|_{W^{s,p}(\Gamma)}$ as small as possible, we call it the
{\em relative strength} of the jump matrix $G$.
  
In the lack of any better understanding of the precise dependence of $\kappa(A)$ on the RHP data $(\Gamma,G)$ we
suggest to use $\|G-I\|_{L^1(\Gamma)}$ as a measure of relative strength: optimizing it led to significant reductions of the
condition number in all of our experiments. However, the deformation algorithm itself will just use that the
measure $d(\Gamma;G)$ can be written as an integral over $\Gamma$, namely in the form
\[
d(\Gamma;G) = \int_\Gamma d(G(z))\, d|z|
\] 
for some function $d: \GL(m,\C) \to [0,\infty)$, which we call the \emph{local weight}. 
 
\subsection*{Preconditioning as a discrete optimization problem}

Since our objective is preconditioning, the relative strength of the jump matrices does not really have to be 
minimized over all equivalent deformations $(\tilde\Gamma,\tilde G)$ of a given RHP $(\Gamma,G)$. For all practical purposes
it suffices to consider just a very coarse, finite set of possible contours, namely paths within a 
planar graph. 

\begin{figure}[tbp]
%row 1
\begin{minipage}{0.45\textwidth}
\begin{center}
	\includegraphics[width=\textwidth]{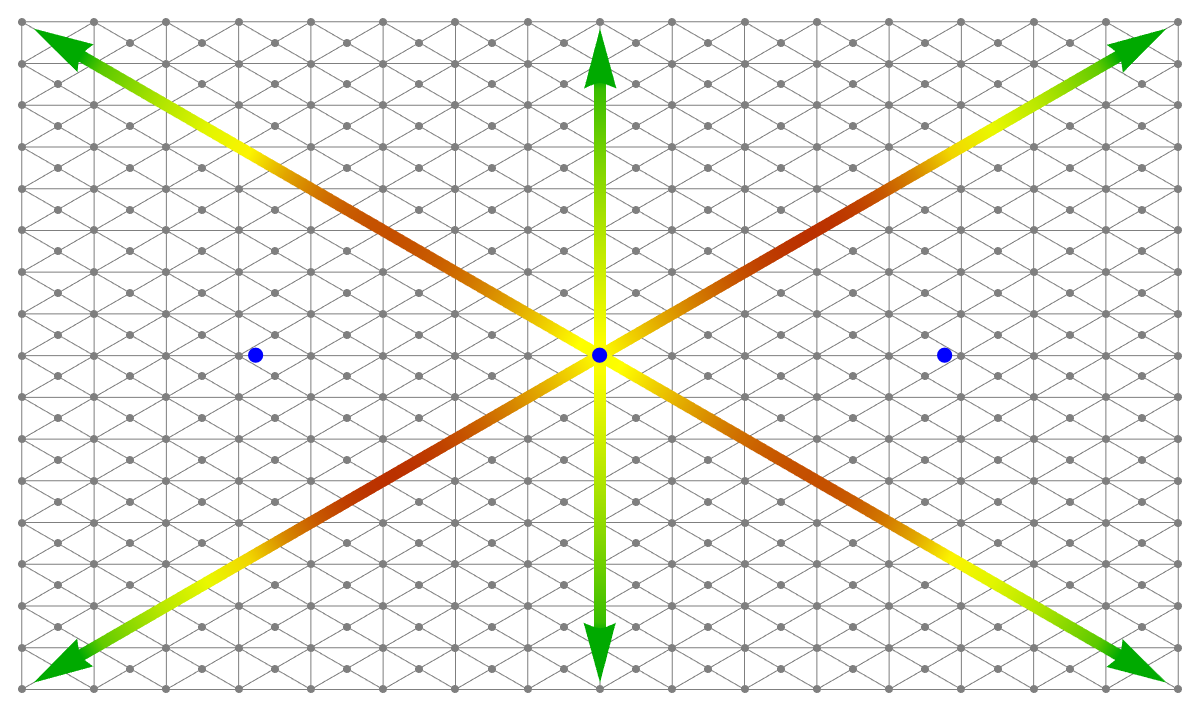}\\
	{\footnotesize a. $(G,\Gamma)$\\
	$\kappa \approx 2.2 \cdot 10^8$}
\end{center}
\end{minipage}
\hfil
\begin{minipage}{0.45\textwidth}
\begin{center}
	\includegraphics[width=\textwidth]{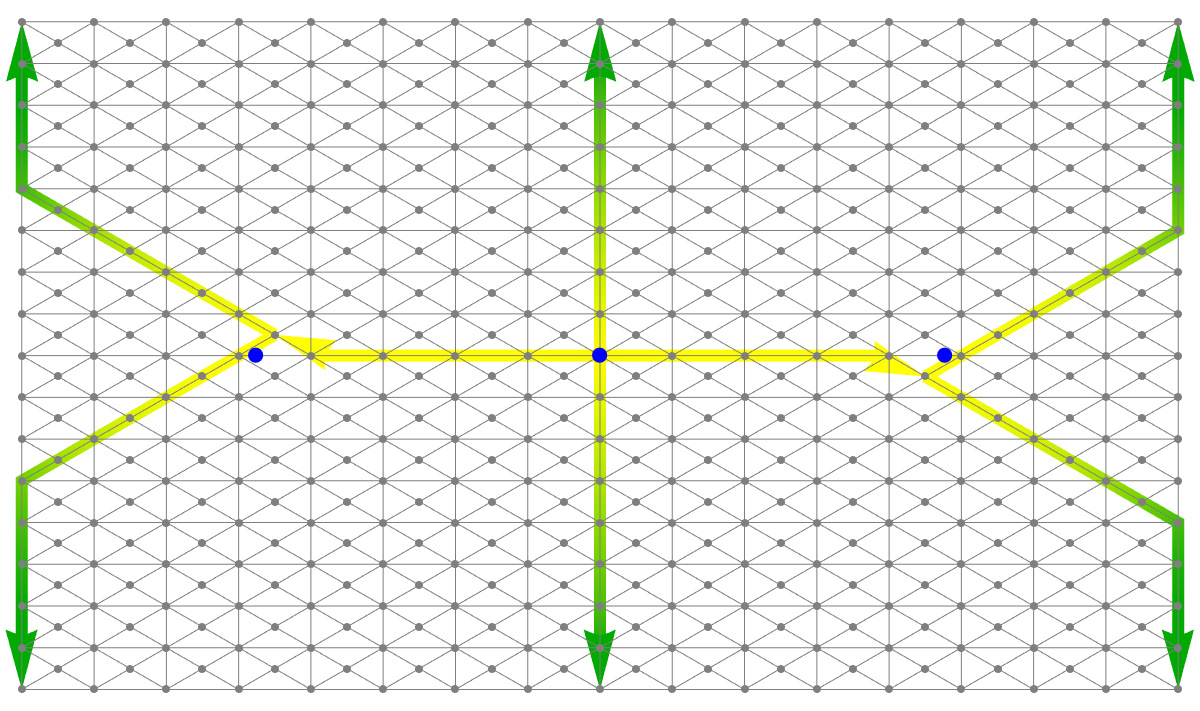}\\
	{\footnotesize b. $(\tilde G^{1},\tilde \Gamma^{1}) = \text{SimpleDeformation}(G,\Gamma)$\\
	 $\kappa \approx 360$ }
\end{center}
\end{minipage}
\\*[4mm]
%row 2
\begin{minipage}{0.45\textwidth}
\begin{center}
	\includegraphics[width=\textwidth]{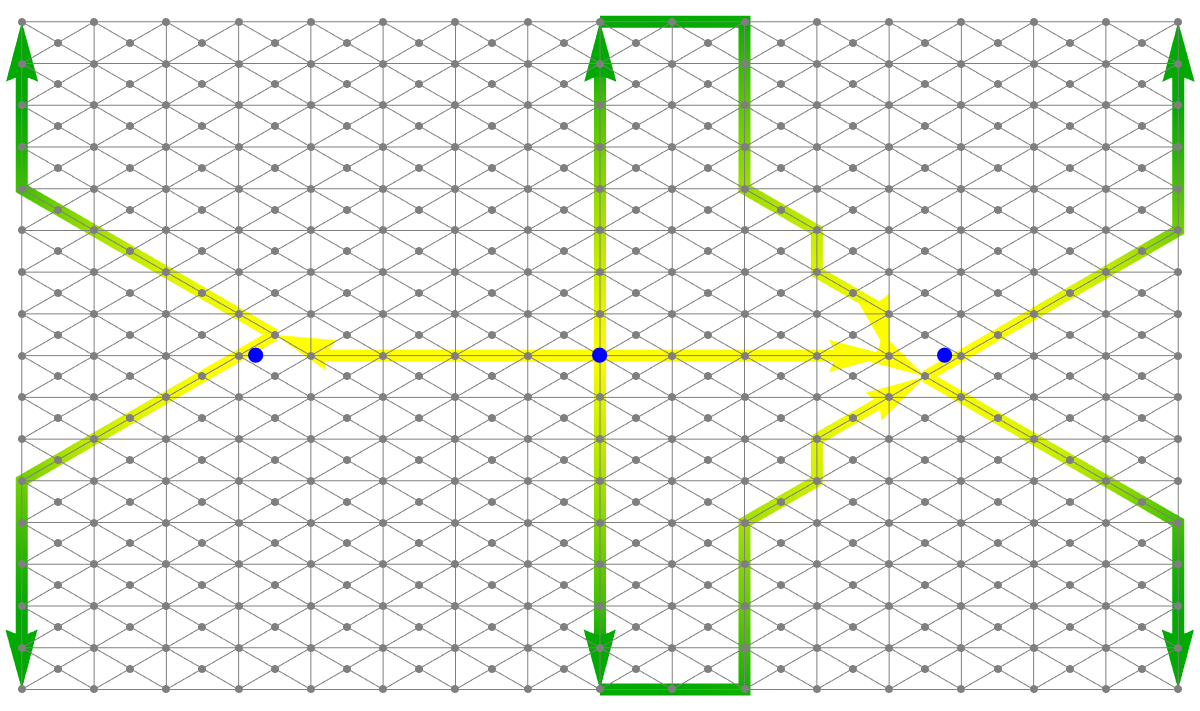}\\
	{\footnotesize c. $(\tilde G^{2},\tilde \Gamma^{2})= \text{LensingDeformation}(\tilde G^1,\tilde \Gamma^1)$\\
	 $\kappa \approx 250$}
\end{center}
\end{minipage}
\hfil
\begin{minipage}{0.45\textwidth}
\begin{center}
	\includegraphics[width=\textwidth]{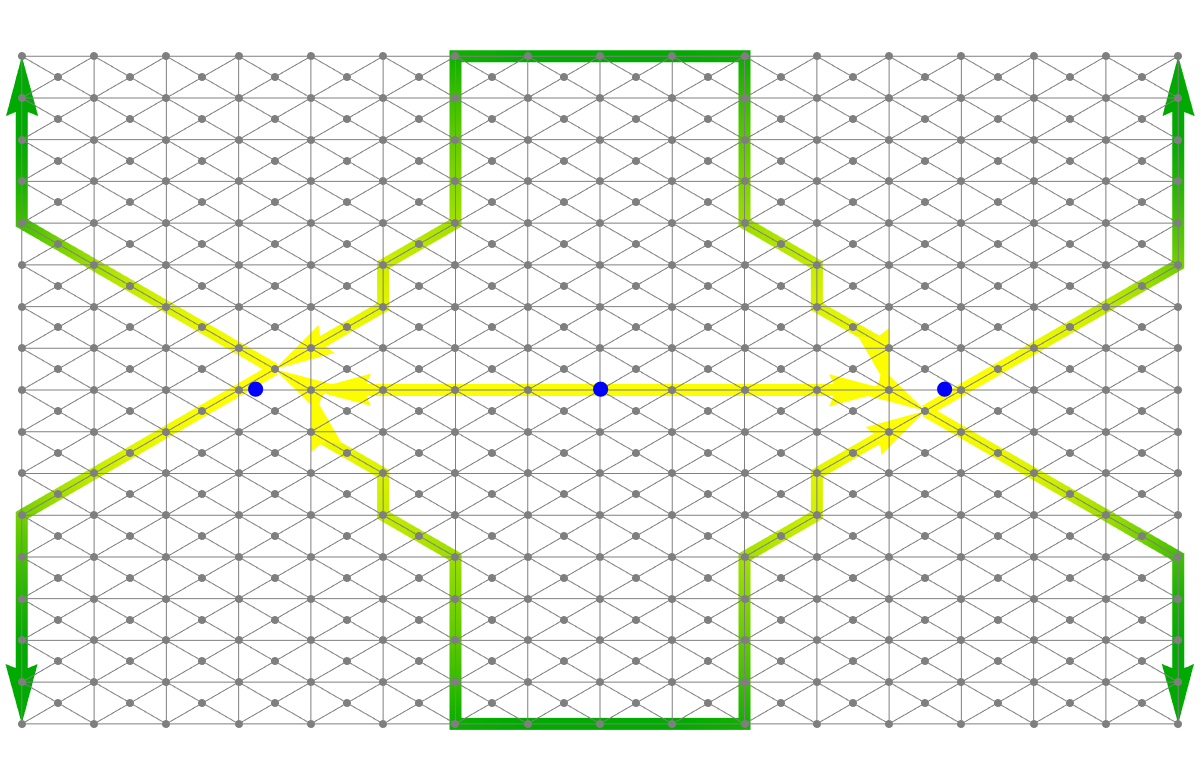}\\
	{\footnotesize d. $(\tilde G^{3},\tilde \Gamma^{3}) = \text{LensingDeformation}(\tilde G^2,\tilde \Gamma^2)$\\
 	$\kappa \approx 140$}
\end{center}
\end{minipage}

\caption{Application of \textsc{SimpleDeformation} and \textsc{LensingDeformation} of §\protect\ref{sec:algorithm} to the Painlevé II RHP with $s_1=1$, $s_2=2$ and $x=-10$. On the top left is the original contour of this RHP and the other contours are deformed versions of it. All contours have been calculated on a
$17 \times 17$ grid. The color encodes the magnitude of $\|G(z)-\I\|_F$, with green = $10^{-16}$, yellow = $1$, and 
red = $10^{4}$. The blue dots indicate the origin $z=0$ and the stationary points of the phase function $\theta$ at $z=\pm \sqrt{-x}/2$. The reduction of the condition number from (a) to (d) corresponds to an accuracy gain of about six digits.}
\label{fig:results}
\end{figure}

The basic idea is as follows: first, we restrict the problem to a bounded region of the complex plane and
embed the part of the contour $\Gamma$ belonging to that region as paths into a coarse, grid-like planar graph $g = (V,E)$
(see Fig.~\ref{fig:results}.a for an example of the Painlevé II RHP: because of a super exponential decay
as $z\to \infty$ along each of the rays, $G-I$ is already 
a computer zero outside the indicated rectangle). 

Second, for each $j$, the analytic continuation $G_j$ of the jump data on $\Gamma_j$
turns the graph $g$ into an edge-weighted graph $g_j$ by using the (edge) weights
\[
d_j(e) = \int_e d(G_j(z))\, d|z| \qquad (e \in E).
\]

Last, we replace $\Gamma_j$ (within the bounded domain) by the \emph{shortest} path (with the same endpoints as $\Gamma_j$)
with respect to $g_j$ subject to the following constraint: the thus deformed RHP must be equivalent to the original one.

It is this latter constraint which adds to the algorithmic difficulty of the problem: the $\Gamma_j$ cannot be optimized independent
of each other. We will address this problem by a greedy strategy: the largest contribution to the weight
constraints the admissible paths of the second largest one and so on; this will be accomplished by modifying the
underlying graphs $g_j$ in the corresponding order.  

Fig.~\ref{fig:results}.b shows the result of such an algorithmic deformation for the Painlevé II RHP ($s_1=1$, $s_2=2$, $x=-10$): the condition number is reduced by about six orders of magnitude. Further improvement is
possible by performing a ``lensing'' deformation, that is, by introducing multiple edges based on a factorization of $G$ 
(see §\ref{subsec:algoLensing}). The results of two such steps are shown in Fig.~\ref{fig:results}.c and d 
(more steps would not pay off). Though the improvement of the condition number is more modest in these two steps, it is instructive to compare the
algorithmic contour in Fig.~\ref{fig:results}.d with the manual construction of Olver and Trogdon \cite{1205.5604} shown in Fig.~\ref{fig:olver}. 

Fig.~\ref{fig:cond} compares, for varying values of $x$, the condition number of the original contour with that of the deformed contour optimized by the
greedy algorithm of §\ref{sec:algorithm}: a uniform stabilization by preconditioning is clearly visible.

\begin{figure}[tbp]
\begin{minipage}{0.9\textwidth}
\begin{center}
	\includegraphics[width=\textwidth]{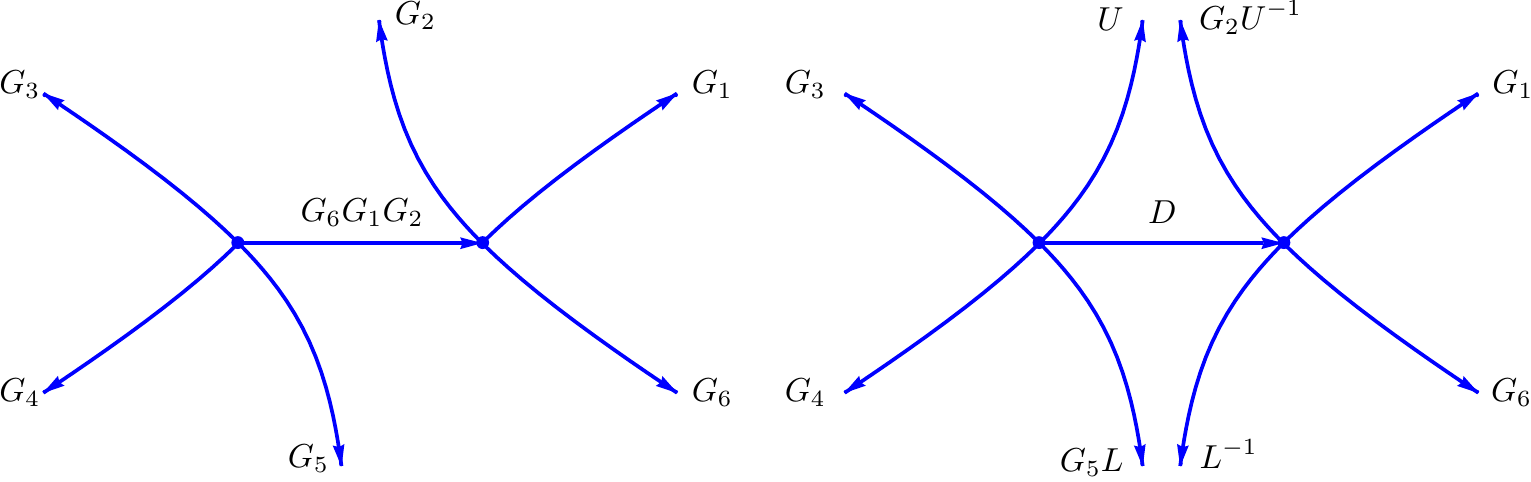}
\end{center}
\end{minipage}
\caption{Some manual constructions for the Painlevé II RHP taken from Olver and Trogdon \protect\cite[p.~20]{1205.5604}. Left: Deformation along the paths of steepest descent; right: deformation after lensing. The contours
bifurcate at the stationary points of the phase function $\theta$.}
\label{fig:olver}
\end{figure}

%Our main goal is to get the condition number of this problem as small as possible by deforming the contour through an algorithm. Deforming means that we change $\Gamma$ in a way such that the solution of the original RHP can be calculated from the solution of the deformed RHP. 
%To achieve this we propose an algorithm similar to the one used in \cite{1107.0498v2} for calculating optimal contours for cauchy integrals. The basic idea of this algorithm is to turn the problem of finding a contour with a lower condition number into an optimization problem on graphs.
%As has been shown in \cite{1205.5604}, the condition number of RHPs can effectively be reduced by applying nonlinear steepest descent (NLS) method. NLS is about deforming the arcs of $\Gamma$ in a way such that $G$ decays to $\I$ as fast as possible along $\Gamma$. This in turn is about the same as finding a shortest path in a graph $g = (V,E)$ if its weights are set to 
%\[ 
%	d(e) = \| G - \I \|_e \quad \forall e \in E.
%\]
%The subscript $e$ of the norm denotes that we integrate $G-\I$ along the edge $e$.
%Following this idea, we propose the algorithm below for optimizing contours of RHPs.
%\begin{itemize}
%		\item create a graph $g_i$ over a part of $\C$ for each arc $\Gamma_i$ of $\Gamma$
%		\item map $\Gamma_i$ to a path in $g_i$
%		\item set the weight of all edges $e$ in $g_i$ to $\| G_i - \I \|_e$
%		\item replace all $\Gamma_i$ with their shortest path in $g_i$
%		\item transform the paths back into a RHP
%\end{itemize}

\subsection*{Outline of the paper}

In §\ref{sec:deform} we discuss the two admissible deformations of RHPs that will be considered in this paper:
simple deformations of contours and lensing deformations based on factorizations of the jump matrix $G$.
We address the question of how to match the topological constraints of such deformations in the planar graphs
attached to each part $\Gamma_j$ of the contour.
In §\ref{sec:algorithm} we give an in-depth description (with pseudo code) of the greedy algorithm that 
aims at optimizing these deformations. Important steps are illustrated for the Painlevé II RHP.
Further details of the implementation are discussed in §\ref{sec:implDetails}.

\section{Admissible Deformations of Riemann--Hilbert Problems}\label{sec:deform}
We briefly recall two of the deformations that can be applied to RHPs. A more detailed description can be found in \cite{fokas:2006:ptr}.

\subsection{Simple Deformations}
\begin{figure}[ht]
	\begin{tikzpicture}
	[rhcontour/.style={->,>=latex',thick},
	 solution/.style={black!60!white}]
	% points on the original contour Gamma
	\coordinate (A) at (0,0);
	\node (D) at (4,2) {$\Gamma$};
	\coordinate (B) at ($(A)!0.2!(D)$);
	\coordinate (C) at ($(A)!0.7!(D)$);
	% points on the deformed contour
	\path let \p1=(5,0) in
		coordinate (A') at ($(A)+(\p1)$)
		coordinate (B') at ($(B)+(\p1)$)
		coordinate (C') at ($(C)+(\p1)$)
		node (D') at ($(D)+(\p1)$) {$\tilde{\Gamma}$};

	\node[anchor=south east,minimum width=4cm] (L) at (D) {undeformed};
	%
	% undeformed contour
	%

	% labels
	\node[solution] (PP) at ($(A)!0.25!(D)!0.2!90:(D)$) {$\Phi$};
	\node[solution] (PM) at ($(A)!0.25!(D)!0.2!-90:(D)$) {$\Phi$};

	% contour parts
	\draw[rhcontour] (A) -- (D) 
		node[midway,below,anchor=west,inner xsep=10] {$G$}
		node[very near start,above,anchor=south east] {$+$}
		node[very near start,below,anchor=north west] {$-$};
	\draw[rhcontour] (A) -- (B);
	\draw[rhcontour] (B) -- (C);
	\draw[rhcontour] (C) -- (D);

	% \cong
	%\node (eq) at ($(D)!0.5!(A')$) {$\cong$};

	%
	% deformed contour
	%
	% labels
	\node[solution] (PP') at ($(A')!0.2!(D')!0.2!90:(D')$) {$\Phi$};
	\node[solution] (PM') at ($(A')!0.2!(D')!0.2!-90:(D')$) {$\Phi$};

	\node[anchor=south east,minimum width=4cm] (L) at (D') {deformed};

	% contour parts
	% background between both pathes
	\draw[fill=red,opacity=0.9] (B') .. controls ($(B')!1!-45:(C')$) and ($(C')!1!90:(B')$) .. (C') -- (B') node[midway,above,red,opacity=0.9] {$\Omega$};
	\draw[rhcontour] (A') -- (B');
	\draw[rhcontour] (C') -- (D')
	  node[midway,above,anchor=east,inner xsep=10] {$G$};
	\draw[rhcontour,out=0,in=-90] (B') .. controls ($(B')!1!-45:(C')$) and ($(C')!1!90:(B')$) .. (C');
	\draw[dotted,thick] (B') -- (C')
	  node[solution,black,midway,below,anchor=north west,inner xsep=10] {$\Phi G$};
\end{tikzpicture}
	\caption{A simple deformation.}
	\label{fig:steepestdescent}
\end{figure}
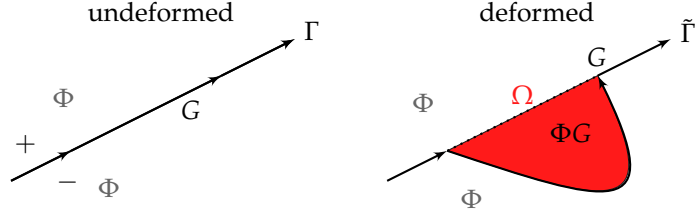

Fig.~\ref{fig:steepestdescent} shows an example of such a deformation. In general, simple deformations allow to continuously move a contour part in the complex plane (thereby covering a region $\Omega$) as long as the following conditions are satisfied:

\begin{itemize}
	\item[(i)] $\tilde\Gamma$ does not cross other parts of $\Gamma$,\\*[-3mm]
	\item[(ii)] $\Omega$ does not contain any other contour parts,\\*[-3mm]
	\item[(iii)] $G$ has a holomorphic continuation in $\Omega$.
\end{itemize}

\noindent
Then, the deformed RHP in Fig.~\ref{fig:steepestdescent} is solved by the function
\[
   \tilde\Phi = \left\{
     \begin{array}{lr}
       \Phi G & : x \in \Omega,\\*[2mm]
       \Phi & : x \notin \Omega.
     \end{array}
   \right.
\]
Conditions (i)-(iii) can be mapped to graph-constrained deformations as follows:

Condition (i) can be handled by {\em splitting} a graph as shown in Fig.~\ref{fig:graphsplit}:
if a path $p$ corresponding to a part of a contour is given, like the path highlighted in blue, 
we duplicate the vertices of $p$ and change all edges on the right side of $p$ so that they are 
connected to the newly created vertices but not to the vertices of $p$ itself. 
This way no path in the graph can cross $p$ anymore. We will use $g[p_1,p_2,\dots]$ to denote a graph $g$ which has been split in this fashion along the paths $p_1,p_2,\dots$. 

Condition (ii) is difficult to be built into the structure of a graph {\em a priori}, but it is easy to check for it
{\em a posteriori}: the circle composed by $\Gamma_i$ and $\tilde\Gamma_i$ should not enclose an endpoint of another arc. If violated, the algorithm simply stops (this never happened in our experiments; dealing with such a situation would
require to break the deformations into smaller pieces).

Condition (iii) can be handled by removing those regions from the graph where $G$ does not have a holomorphic continuation. 

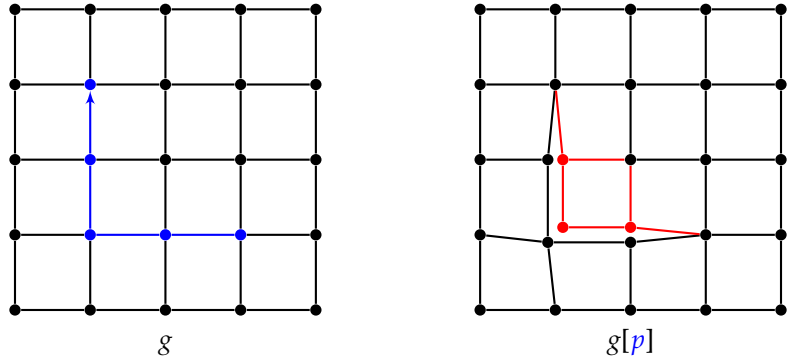
\begin{figure}[ht]
%row 1
\begin{minipage}{0.45\textwidth}
\begin{center}
	\begin{tikzpicture}
[vertex/.style={circle,minimum size=1.5mm,inner sep=0pt,fill=black},
edge/.style={thick},
splitVertex/.style={fill=red},
splitLabel/.style={red},
highlightLabel/.style={blue},
defaultLabel/.style={black},
newEdge/.style={red},
changedEdge/.style={red},
end/.style={->,>=latex'},
highlightEdge/.style={blue},
highlightVertex/.style={blue},
]\node[vertex,] (1) at (1. , 1.) {}; 
\node[vertex,] (2) at (1. , 2.) {}; 
\node[vertex,] (3) at (1. , 3.) {}; 
\node[vertex,] (4) at (1. , 4.) {}; 
\node[vertex,] (5) at (1. , 5.) {}; 
\node[vertex,] (6) at (2. , 1.) {}; 
\node[vertex,highlightVertex,] (7) at (2. , 2.) {}; 
\node[vertex,highlightVertex,] (8) at (2. , 3.) {}; 
\node[vertex,highlightVertex,] (9) at (2. , 4.) {}; 
\node[vertex,] (10) at (2. , 5.) {}; 
\node[vertex,] (11) at (3. , 1.) {}; 
\node[vertex,highlightVertex,] (12) at (3. , 2.) {}; 
\node[vertex,] (13) at (3. , 3.) {}; 
\node[vertex,] (14) at (3. , 4.) {}; 
\node[vertex,] (15) at (3. , 5.) {}; 
\node[vertex,] (16) at (4. , 1.) {}; 
\node[vertex,highlightVertex,] (17) at (4. , 2.) {}; 
\node[vertex,] (18) at (4. , 3.) {}; 
\node[vertex,] (19) at (4. , 4.) {}; 
\node[vertex,] (20) at (4. , 5.) {}; 
\node[vertex,] (21) at (5. , 1.) {}; 
\node[vertex,] (22) at (5. , 2.) {}; 
\node[vertex,] (23) at (5. , 3.) {}; 
\node[vertex,] (24) at (5. , 4.) {}; 
\node[vertex,] (25) at (5. , 5.) {}; 
\draw[edge,](1) -- (2) ; 
\draw[edge,](1) -- (6) ; 
\draw[edge,](2) -- (3) ; 
\draw[edge,](2) -- (7) ; 
\draw[edge,](3) -- (4) ; 
\draw[edge,](3) -- (8) ; 
\draw[edge,](4) -- (5) ; 
\draw[edge,](4) -- (9) ; 
\draw[edge,](5) -- (10) ; 
\draw[edge,](6) -- (7) ; 
\draw[edge,](6) -- (11) ; 
\draw[edge,highlightEdge,](7) -- (8) ; 
\draw[edge,highlightEdge,](12) -- (7) ; 
\draw[edge,highlightEdge,end](8) -- (9) ; 
\draw[edge,](8) -- (13) ; 
\draw[edge,](9) -- (10) ; 
\draw[edge,](9) -- (14) ; 
\draw[edge,](10) -- (15) ; 
\draw[edge,](11) -- (12) ; 
\draw[edge,](11) -- (16) ; 
\draw[edge,](12) -- (13) ; 
\draw[edge,highlightEdge,](17) -- (12) ; 
\draw[edge,](13) -- (14) ; 
\draw[edge,](13) -- (18) ; 
\draw[edge,](14) -- (15) ; 
\draw[edge,](14) -- (19) ; 
\draw[edge,](15) -- (20) ; 
\draw[edge,](16) -- (17) ; 
\draw[edge,](16) -- (21) ; 
\draw[edge,](17) -- (18) ; 
\draw[edge,](17) -- (22) ; 
\draw[edge,](18) -- (19) ; 
\draw[edge,](18) -- (23) ; 
\draw[edge,](19) -- (20) ; 
\draw[edge,](19) -- (24) ; 
\draw[edge,](20) -- (25) ; 
\draw[edge,](21) -- (22) ; 
\draw[edge,](22) -- (23) ; 
\draw[edge,](23) -- (24) ; 
\draw[edge,](24) -- (25) ; 
\end{tikzpicture}\\
	{$g$\\}
\end{center}
\end{minipage}
\hfil
\begin{minipage}{0.45\textwidth}
\begin{center}
	\begin{tikzpicture}
[vertex/.style={circle,minimum size=1.5mm,inner sep=0pt,fill=black},
edge/.style={thick},
splitVertex/.style={fill=red},
splitLabel/.style={red},
highlightLabel/.style={blue},
defaultLabel/.style={black},
newEdge/.style={red},
changedEdge/.style={red},
end/.style={->,>=latex'},
highlightEdge/.style={blue},
highlightVertex/.style={blue},
]\node[vertex,] (1) at (1. , 1.) {}; 
\node[vertex,] (2) at (1. , 2.) {}; 
\node[vertex,] (3) at (1. , 3.) {}; 
\node[vertex,] (4) at (1. , 4.) {}; 
\node[vertex,] (5) at (1. , 5.) {}; 
\node[vertex,] (6) at (2. , 1.) {}; 
\node[vertex,] (7) at (1.9 , 1.9) {}; 
\node[vertex,] (8) at (1.9 , 3.) {}; 
\node[vertex,] (9) at (2. , 4.) {}; 
\node[vertex,] (10) at (2. , 5.) {}; 
\node[vertex,] (11) at (3. , 1.) {}; 
\node[vertex,] (12) at (3. , 1.9) {}; 
\node[vertex,] (13) at (3. , 3.) {}; 
\node[vertex,] (14) at (3. , 4.) {}; 
\node[vertex,] (15) at (3. , 5.) {}; 
\node[vertex,] (16) at (4. , 1.) {}; 
\node[vertex,] (17) at (4. , 2.) {}; 
\node[vertex,] (18) at (4. , 3.) {}; 
\node[vertex,] (19) at (4. , 4.) {}; 
\node[vertex,] (20) at (4. , 5.) {}; 
\node[vertex,] (21) at (5. , 1.) {}; 
\node[vertex,] (22) at (5. , 2.) {}; 
\node[vertex,] (23) at (5. , 3.) {}; 
\node[vertex,] (24) at (5. , 4.) {}; 
\node[vertex,] (25) at (5. , 5.) {}; 
\node[vertex,splitVertex,] (26) at (3. , 2.1) {}; 
\node[vertex,splitVertex,] (27) at (2.1 , 2.1) {}; 
\node[vertex,splitVertex,] (28) at (2.1 , 3.) {}; 
\draw[edge,](1) -- (2) ; 
\draw[edge,](1) -- (6) ; 
\draw[edge,](2) -- (3) ; 
\draw[edge,](2) -- (7) ; 
\draw[edge,](3) -- (4) ; 
\draw[edge,](3) -- (8) ; 
\draw[edge,](4) -- (5) ; 
\draw[edge,](4) -- (9) ; 
\draw[edge,](5) -- (10) ; 
\draw[edge,](6) -- (7) ; 
\draw[edge,](6) -- (11) ; 
\draw[edge,](7) -- (8) ; 
\draw[edge,](7) -- (12) ; 
\draw[edge,](8) -- (9) ; 
\draw[edge,](9) -- (10) ; 
\draw[edge,](9) -- (14) ; 
\draw[edge,changedEdge,](9) -- (28) ; 
\draw[edge,](10) -- (15) ; 
\draw[edge,](11) -- (12) ; 
\draw[edge,](11) -- (16) ; 
\draw[edge,](12) -- (17) ; 
\draw[edge,](13) -- (14) ; 
\draw[edge,](13) -- (18) ; 
\draw[edge,changedEdge,](13) -- (26) ; 
\draw[edge,changedEdge,](13) -- (28) ; 
\draw[edge,](14) -- (15) ; 
\draw[edge,](14) -- (19) ; 
\draw[edge,](15) -- (20) ; 
\draw[edge,](16) -- (17) ; 
\draw[edge,](16) -- (21) ; 
\draw[edge,](17) -- (18) ; 
\draw[edge,](17) -- (22) ; 
\draw[edge,changedEdge,](17) -- (26) ; 
\draw[edge,](18) -- (19) ; 
\draw[edge,](18) -- (23) ; 
\draw[edge,](19) -- (20) ; 
\draw[edge,](19) -- (24) ; 
\draw[edge,](20) -- (25) ; 
\draw[edge,](21) -- (22) ; 
\draw[edge,](22) -- (23) ; 
\draw[edge,](23) -- (24) ; 
\draw[edge,](24) -- (25) ; 
\draw[edge,newEdge,](26) -- (27) ; 
\draw[edge,newEdge,](27) -- (28) ; 
\end{tikzpicture}\\
	{$g$[\textcolor{blue}{$p$}]\\}
\end{center}
\end{minipage}
%\\*[4mm]
%%row 2
%\begin{minipage}{0.45\textwidth}
%\begin{center}
%	\input{plots/graphSplitBoundaryLeft.tex}\\
%	{g\\}
%\end{center}
%\end{minipage}
%\hfil
%\begin{minipage}{0.45\textwidth}
%\begin{center}
%	\input{plots/graphSplitBoundaryRight.tex}\\
%	{g[\textcolor{blue}{(9,10,6,7)}]\\}
%\end{center}
%\end{minipage}
\caption{Illustration of split graphs. On the left side is the original graph $g$ with a path $p$ (highlighted in blue)
along which it is about to be split. The split graph $g[p]$ is on the right side, with all vertices and edges that
 have been changed or created highlighted in red. For a clear visualization the split in the graph has been enlarged
 by moving the vertex positions; in the actual graphs used by the algorithm the duplicated vertices stay at exactly the same position. All graphs are undirected, the arrow at the end of the blue path just indicates its orientation.}
\label{fig:graphsplit}
\end{figure}

\subsection{Multiple Deformations and Factorization: Lensing}
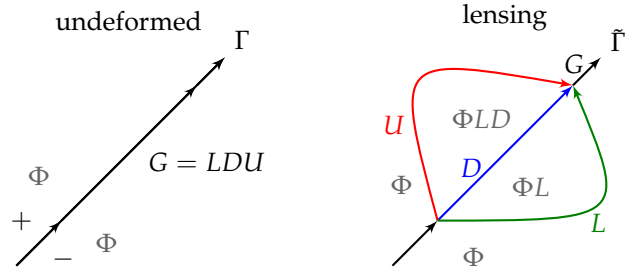
\begin{figure}[ht]
	\begin{tikzpicture}
	[rhcontour/.style={->,>=latex',thick},
	 solution/.style={black!60!white}]
	% points on the original contour Gamma
	\coordinate (A) at (0,0);
	\node (D) at (3,3) {$\Gamma$};
	\coordinate (B) at ($(A)!0.2!(D)$);
	\coordinate (C) at ($(A)!0.8!(D)$);
	% points on the deformed contour
	\path let \p1=(5,0) in
		coordinate (A') at ($(A)+(\p1)$)
		coordinate (B') at ($(B)+(\p1)$)
		coordinate (C') at ($(C)+(\p1)$)
		node (D') at ($(D)+(\p1)$) {$\tilde\Gamma$};
	%
	% undeformed contour
	%
	\node[anchor=south east,minimum width=3cm] (L) at (D) {undeformed};
 
	% labels
	\node[solution] (PP) at ($(A)!0.25!(D)!0.2!90:(D)$) {$\Phi$};
	\node[solution] (PM) at ($(A)!0.25!(D)!0.2!-90:(D)$) {$\Phi$};

	% contour parts
	\draw[rhcontour] (A) -- (D) 
		node[midway,below,anchor=west,inner xsep=10] {$G = LDU$}
		node[very near start,above,anchor=south east] {$+$}
		node[very near start,below,anchor=north west] {$-$};
	\draw[rhcontour] (A) -- (B);
	\draw[rhcontour] (B) -- (C);
	\draw[rhcontour] (C) -- (D);

	%
	% lensing deformation
	%
	
	% labels
	\node[solution] (PP') at ($(A')!0.2!(D')!0.2!90:(D')$) {$\Phi$};
	\node[solution] (PM') at ($(A')!0.2!(D')!0.2!-90:(D')$) {$\Phi$};

	\node[anchor=south east,minimum width=3cm] (L) at (D') {lensing};

	% contour parts
	\draw[rhcontour] (A') -- (B');
	\draw[rhcontour] (C') -- (D')
		node[near end,above,anchor=east] {$G$};
	\draw[rhcontour,blue] (B') -- (C')
		node[near start,above] {$D$}
		node[solution,near end,above,anchor=east,inner xsep=10] {$\Phi LD$}
		node[solution,near start,below,anchor=west,inner xsep=15] {$\Phi L$};
	\draw[rhcontour,out=90,in=180,red] (B') .. controls ($(B')!0.9!60:(C')$) .. (C')
		node[near start,above,anchor=east] {$U$};
	\draw[rhcontour,out=0,in=-90,green!50!black] (B') .. controls ($(B')!1!-45:(C')$) .. (C')
		node[midway,below] {$L$};

\end{tikzpicture}
	\caption{A lensing deformation.}
	\label{fig:lensing}
\end{figure}
Fig.~\ref{fig:lensing} shows an example of such a deformation. To initialize, several copies of a contour part are created at one and the same location, where each
copy corresponds to a factor of a given multiplicative decomposition of the jump matrix $G$. We
call these copies the \emph{factors} of this part of the contour $\Gamma$. These factors are then moved around in the complex plane 
subject to conditions (i)-(iii) and, additionally, the following condition:

\begin{itemize}
	\item[(iv)] the mutual orientation of the factors must be preserved.
\end{itemize}

For example, in Fig.~\ref{fig:lensing} the order of the decomposition $G = LDU$ requires that the factor $U$ is to the left of 
the factor $D$ and that the factor $D$ is to the left of the factor $L$. To preserve this orientation in our deformation algorithm, we calculate the shortest path for just one of the factors. For the other factors we use a 
modification of the shortest enclosing circle algorithm of Provan \cite{provan89}, see §\ref{subsec:algoLensing}.

\section{The Greedy Algorithm}\label{sec:algorithm}

\subsection{Notation}
\begin{itemize}
	\item $d_e$ : weight of the edge $e$\\*[-2mm]
	\item $P_1 \join P_2$ : path $P_1$ joined with path $P_2$\\*[-2mm]
	\item $\overleftarrow{P}$ : reversed path $P$\\*[-2mm]
	\item $P[u,v]$ : subpath from vertex $u$ to vertex $v$ within the path $P$\\*[-2mm]
	\item $\spa(g,u,v)$ : shortest path from vertex $u$ to $v$ in the weighted graph $g$\\*[-2mm]
	\item $\interior(W)$ : homological interior of a closed walk $W$, that is, all vertices $v$ of the graph that have winding number $\ind(W,v) = \pm 1$ w.r.t. $W$\\*[-2mm]
	\item $p_- / p_+$ : path on the left/right side of the split along $p$ in $g[p]$ (see Fig.~\ref{fig:graphsplit})
\end{itemize}

\subsection{Optimized Simple Deformations}\label{subsec:algoSteepestDescent}
\begin{algorithm}[ht]
	\caption{Optimized Simple Deformation}\label{algo:optPath}
	\begin{algorithmic}[1]
		\Procedure{SimpleDeformation}{$G,\Gamma$}
		\State $n = |\Gamma|$ \Comment{$n$ is the number of contour parts}
		\State $P = ()$ 			\Comment{fixed new paths}
		\State $p = ()$ 			\Comment{candidates for new paths}
		\State $F = ()$ 			\Comment{already processed contour parts}
		\State $Q = (1,\dots,n)$ 		\Comment{unprocessed contour parts}
		\State $g = ()$          		\Comment{graphs corresponding to contour parts}
		\State $g^\star = ()$    		\Comment{initial not split graphs}
		\ForAll{$i \in Q$}
		\State $g_i$ = graph with edge weights $d_e = \| G_i - I \|_{L^1(e)}$  \label{line:graphcreation}
		\State $v^l_i$ = vertex in $g_i$ nearest to left endpoint of $\Gamma_i$ \label{line:endpoint}
		\State $v^r_i$ = vertex in $g_i$ nearest to right endpoint of $\Gamma_i$
		\EndFor
		\State $g^\star = g$
		\While{$Q \ne  ()$}
			\ForAll{$i \in Q$}
				\State $p_i = \text{sp}(g_i,v^l_i,v^r_i)$ \Comment{shortest path from $v^l_i$ to $v^r_i$ in $g_i$} \label{line:shortestpaths}
			\EndFor
			\State $i_{\star} = \underset{i \in Q}{\argmax}\,d(p_i)$ \label{line:selectPath}
			\State $P_{i_\star} = p_{i_\star}$
			\State $Q = Q \backslash (i_\star)$
			\ForAll{$i \in F$}
			\Comment{try to improve the path if there is an intersection}
				\If{$P_{i_\star} \cap P_i \ne \emptyset $}
					\State $P = \text{ImproveSharedSubpath}(g^\star,G,P,i,i_\star)$ \label{line:improveCommonSubPath}
				\EndIf
			\EndFor
			\State $F = F \cup (i_\star)$
			\ForAll{$i \in Q$}
				\State $g_i = g^\star_i[P]$ \label{line:splitGraphs}
			\EndFor
		\EndWhile
		\State $(\tilde G,\tilde\Gamma) = \text{MapToRHP}(G,P)$ \Comment{create new contour parts} \label{line:mapToRHP}
		\State \Return{$(\tilde G, \tilde \Gamma)$}
		\EndProcedure
	\end{algorithmic}
\end{algorithm}

\newcommand{\sixpack}[6]{
	\begin{center}
	\begin{minipage}{0.32\textwidth}
		\begin{center}
		\includegraphics[width=\textwidth]{#1}
		\end{center}
	\end{minipage}
	\hfil
	\begin{minipage}{0.32\textwidth}
		\begin{center}
		\includegraphics[width=\textwidth]{#2}
		\end{center}
	\end{minipage}
	\hfil
	\begin{minipage}{0.32\textwidth}
		\begin{center}
		\includegraphics[width=\textwidth]{#3}
		\end{center}
	\end{minipage}\\
	\end{center}
	\vspace{0.1cm}
	% row 2
	\begin{center}
	\begin{minipage}{0.32\textwidth}
		\begin{center}
		\includegraphics[width=\textwidth]{#4}
		\end{center}
	\end{minipage}
	\hfil
	\begin{minipage}{0.32\textwidth}
		\begin{center}
		\includegraphics[width=\textwidth]{#5}
		\end{center}
	\end{minipage}
	\hfil
	\begin{minipage}{0.32\textwidth}
		\begin{center}
		\includegraphics[width=\textwidth]{#6}
		\end{center}
	\end{minipage}\\
	\end{center}
}
\newcommand{\sixpackSimple}[1]{\sixpack{#1-3}{#1-2}{#1-1}{#1-4}{#1-5}{#1-6}}

The idea of Algorithm~\ref{algo:optPath} goes as follows: first (lines 9--13), for each of the contour parts
$\Gamma_j$ and the corresponding jump matrices $G_j$ (which are assumed to have a holomorphic continuation
to the rectangular region supporting the grid), a separate weighted graph $g_j$ with edge weights 
\[
d_e = \int_e d(G_j(z))\, d|z|
\]
\begin{figure}[tbp]
	\sixpackSimple{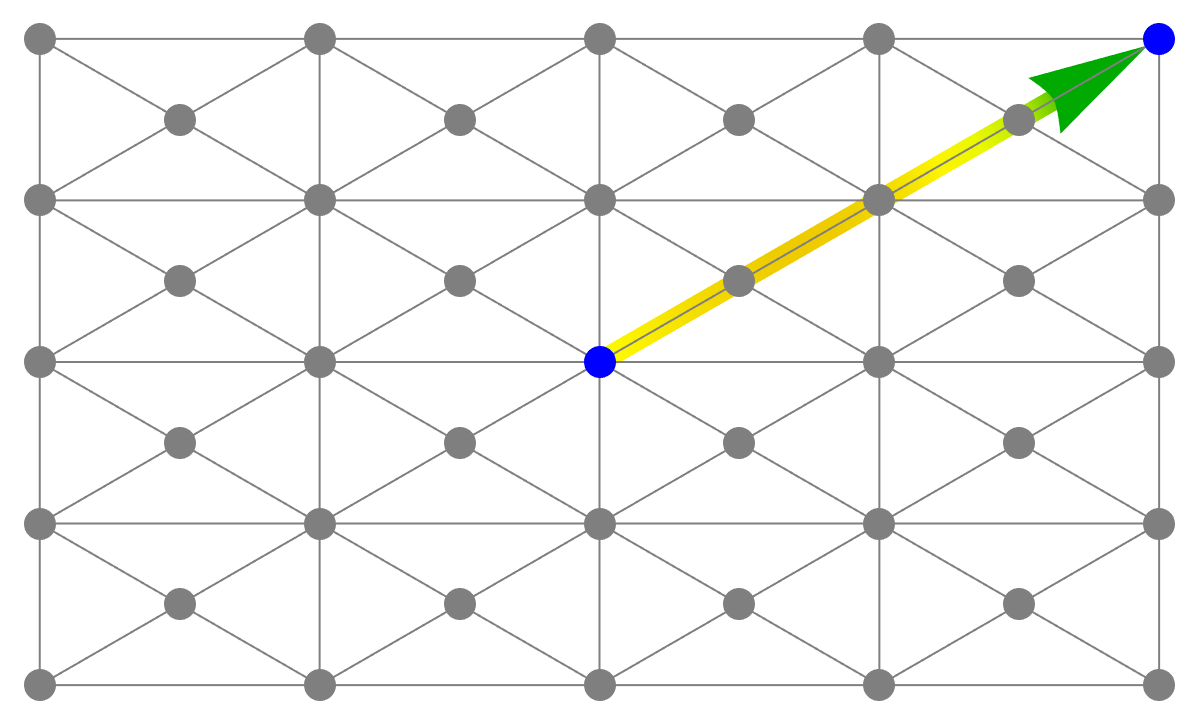}
	\caption{Create separate weighted graphs for each contour part $\Gamma_j$ with weights depending on the
	jump matrix $G_j$ (line \ref{line:graphcreation} in Algorithm~\ref{algo:optPath}).}\label{fig:separate}
\end{figure}%
\begin{figure}[tbp]
	\sixpackSimple{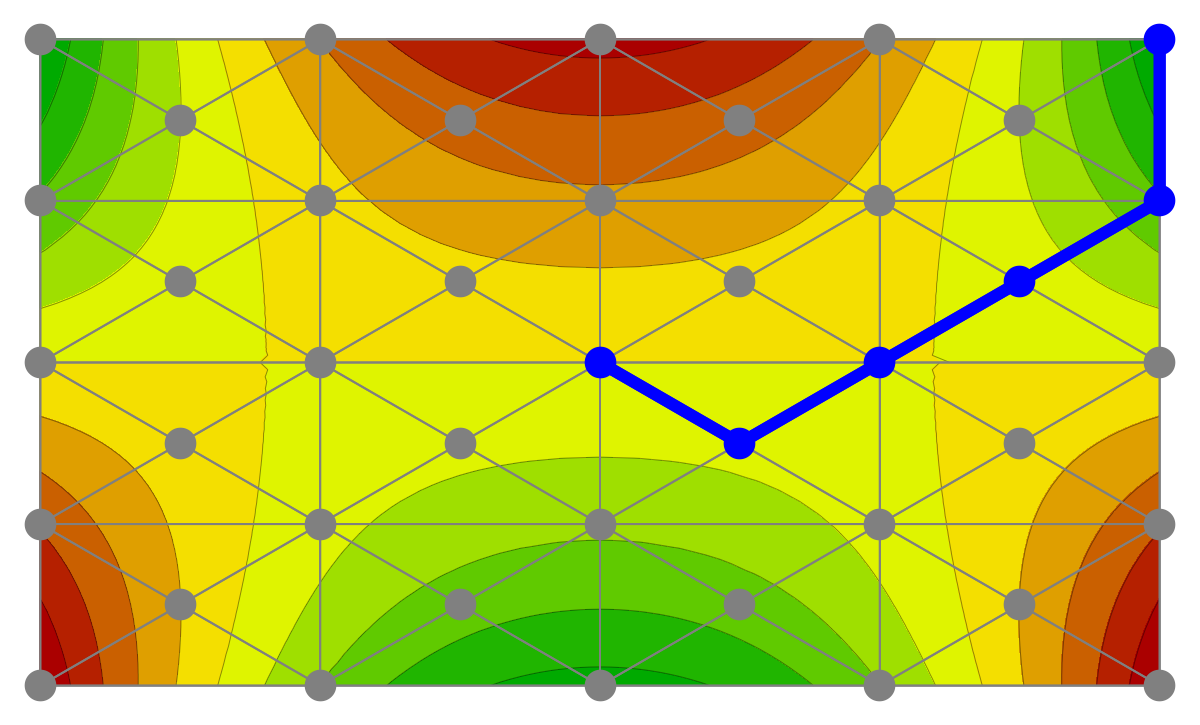}
	\caption{Calculate the shortest paths for each $G_j$ separately, highlighted in blue (line \ref{line:shortestpaths}
	in Algorithm~\ref{algo:optPath}). The one with the largest total weight is shown in magenta (line \ref{line:selectPath}). The color encodes the magnitude of $\|G_j(z)-\I\|_F$, with green = $10^{-16}$, yellow = $1$, and red=$10^{16}$.}\label{fig:firstStep}
\end{figure}%
is created, see Fig.~\ref{fig:separate}. Second (lines 16--18), each $\Gamma_j$ is replaced by a shortest path
that shares the same endpoints, see Fig.~\ref{fig:firstStep}. The thus separately optimized paths, however, will in general not satisfy condition (i) of §\ref{sec:deform}, that is,
they will cross each other. Therefore, some of the paths have to be modified to match this condition, which increases
the corresponding weight. By keeping, third (lines 19--20), the path $P$ of dominant total weight fixed, we restrict such
modifications to the other parts that contribute less to the condition number. By splitting, fourth (lines 28--29),
all graphs along $P$ and repeating (lines 16--18) the calculation of the shortest paths in the split graphs,
we come up with paths that do not cross $P$, see Fig.~\ref{fig:secondStep}. 

\begin{figure}[tbp]
	\sixpackSimple{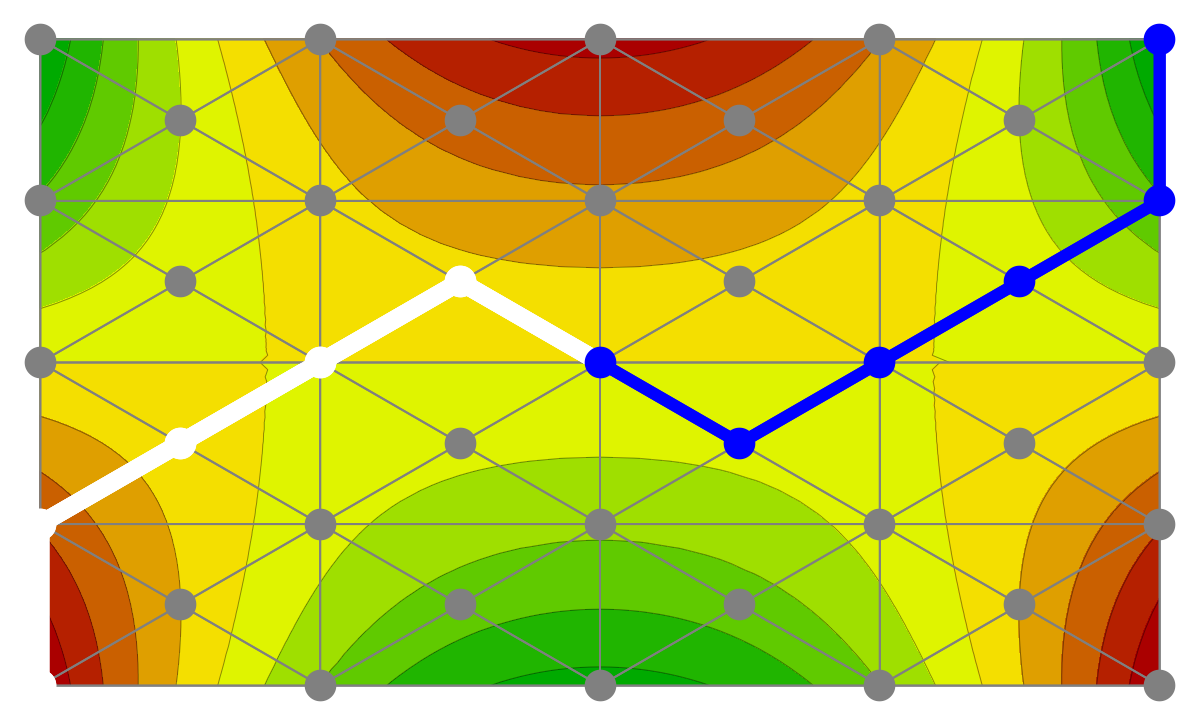}
	\caption{Recalculate the shortest paths (line \ref{line:shortestpaths} in Algorithm~\ref{algo:optPath}) in the
	graphs split along the optimal path of largest weight from Fig.~\ref{fig:firstStep}, here shown in white. The new
	shortest paths are highlighted in blue, the one of maximal weight in magenta.}\label{fig:secondStep}
\end{figure}

\begin{figure}[tbp]
	\sixpackSimple{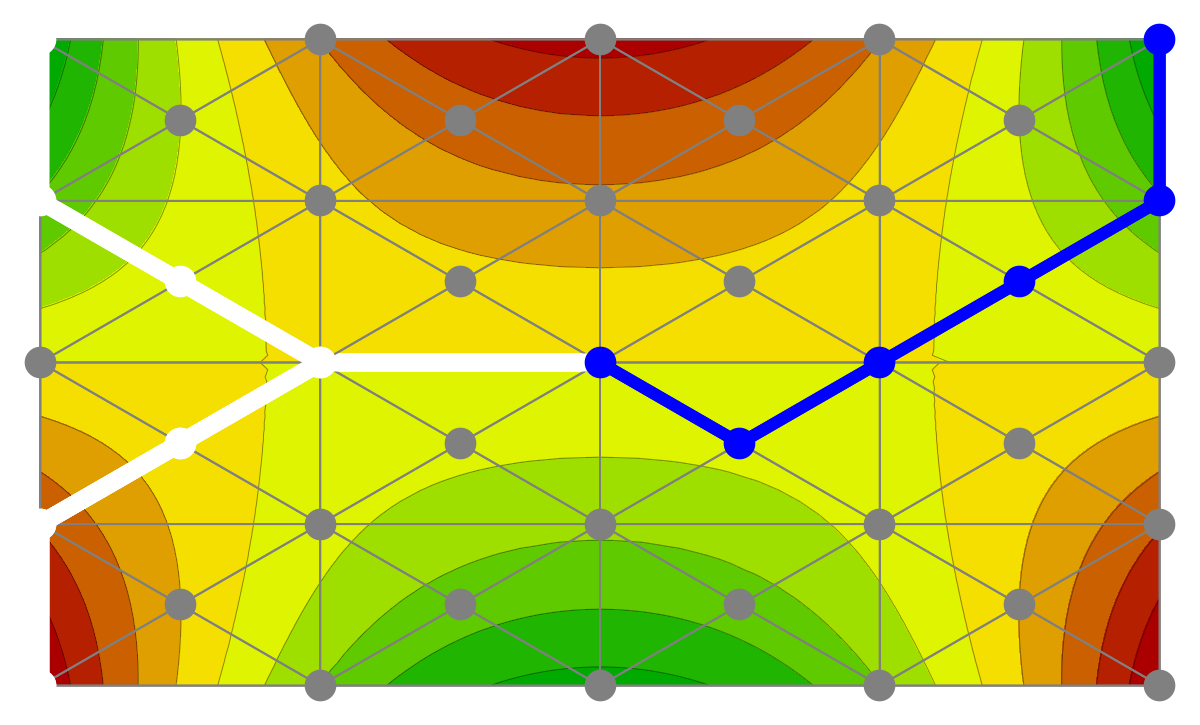}
	\caption{Improvement of the subpath shared by the fixed white path and the magenta path of Fig.~\ref{fig:secondStep}
	(Algorithm~\ref{algo:improvePath}). After this improvement, both paths are fixed and the graphs are split along their union (the white Y-shape contour). From here it should be clear, how Algorithm~\ref{algo:optPath}
	arrives (for a finer grid) at the deformed contours shown in Fig.~\ref{fig:results}.b.}\label{fig:thirdStep}
\end{figure}

This procedure is then repeated, fifth (line 15), until all paths are fixed and, hence, non-crossing. (In each round of this loop another path  gets fixed.) Finally, sixth (line \ref{line:mapToRHP}),
the algorithm constructs the deformed contour data from the just calculated set of paths. 
For paths and subpaths that do not share an edge with another path we simply use the path and the corresponding jump matrix as new contour data. Subpaths which occur in more than one path will be mapped to new contour data
by performing an ``inverse lensing'': the new jump matrix is calculated as the (properly ordered) product of all the jump matrices sharing the that subpath. For example, if the paths $P_i$ and $P_j$ have a common subpath $s$ and $P_i$ is to the left of $P_j$, the procedure {\sc MapToRHP} creates a new contour part $s$ with the jump matrix $G_j G_i$.

In a situation as shown in Fig.~\ref{fig:secondStep}, where the new optimal paths share a subpath with
the already fixed ones, further improvement is possible (lines 23--25) by optimizing the shared subpath with respect
to the weight obtained from combining the corresponding jump matrices (that is, the just mentioned ``inverse lensing''). This procedure (Algorithm~\ref{algo:improvePath}) is schematically illustrated in Fig.~\ref{fig:recursiveImprove}; the application of this procedure to the example of Fig.~\ref{fig:secondStep} is shown in Fig.~\ref{fig:thirdStep}.

%\subsection{Shared Subpath Improvement}\label{subsec:algoImprovePath}
\begin{algorithm}[ht]
	\caption{Improve Paths with a Shared Subpath}\label{algo:improvePath}
	\begin{algorithmic}[1]
		\Procedure{ImproveSharedSubpath}{$g,G,P,i_1,i_2$}
			\State $g' = g_{i_1}$ 		\Comment{create a temporary graph}
			\If{$P_{i_1} \text{ left of } P_{i_2}$}
				\State set weights for edges $e$ of $g'$ to $d(e) = \| G_{i_2} G_{i_1}  - I \|_{L^1(e)}$
			\EndIf
			\If{$P_{i_1} \text{ right of } P_{i_2}$}
				\State set weights for edges $e$ of $g'$ to $d(e) = \| G_{i_1} G_{i_2}  - I \|_{L^1(e)}$
			\EndIf
			\State $s = P_{i_1} \cap P_{i_2}$     \Comment{split $P_{i_1}$ and $P_{i_2}$ into left,common and right subpath} \label{line:restrictionsStart}
			\State $(P^l_{i_1}, s, P^r_{i_1}) = P_{i_1}$
			\State $(P^l_{i_2}, s, P^r_{i_2}) = P_{i_2}$
			\State $P' = P \backslash \{P_{i_1},P_{i_2}\}$
			\State $g' = g'[P',P^l_{i_1} \join \overleftarrow{P^l_{i_2}}, P^r_{i_1} \join \overleftarrow{P^r_{i_2}}]$ \label{line:restrictionsEnd}
			\State $p_\star = sp(g',s_1,s_{-1})$  \label{line:shortestPath} \Comment{$p_\star =$ shortest path between start and end vertex of s}
			\State $P_{i_1} = P^l_{i_1} \cup p_\star \cup P^r_{i_1}$
			\State $P_{i_2} = P^l_{i_2} \cup p_\star \cup P^r_{i_2}$
			\If{$\text{containsCircle}(P_{i_1})$} \label{line:recursionStart}	
				\State $P_{i_1} = \text{dropCircle}(P_{i_1})$
				\State $s = P_{i_1} \cap P_{i_2}$
				\State $P_{i_2} = \text{sp}(g_{i_2}[ P \backslash P_{i_2}],(P_{i_2})_1,(P_{i_2})_{-1})$
				\If{ $P_{1_1} \cap P_{i_2} \ne s$}
					\State $P = \text{ImproveSharedSubpath}(g,G,P,i_1,i_2)$
				\EndIf
			\ElsIf{$\text{containsCircle}(P_{i_2})$}
				\State $P_{i_2} = \text{dropCircle}(P_{i_2})$
				\State $s = P_{i_1} \cap P_{i_2}$
				\State $P_{i_1} = \text{sp}(g_{i_1}[ P \backslash P_{i_1}],(P_{i_1})_1,(P_{i_1})_{-1})$
				\If{ $P_{1_1} \cap P_{i_2} \ne s$}
					\State $P = \text{ImproveSharedSubpath}(g,G,P,i_1,i_2)$
				\EndIf
			\EndIf 			      \label{line:recursionEnd}
			\State \Return{$P$}
		\EndProcedure
	\end{algorithmic}
\end{algorithm}

\begin{figure}[tbp]
	\begin{minipage}{0.45\textwidth}
	\begin{center}
		\begin{tikzpicture}
[vertex/.style={circle,minimum size=1.5mm,inner sep=0pt,fill=black},
edge/.style={thick},
splitVertex/.style={fill=red},
splitLabel/.style={red},
highlightLabel/.style={blue},
defaultLabel/.style={black},
newEdge/.style={red},
changedEdge/.style={red},
leftEnd/.style={<-,>=latex'},
rightEnd/.style={->,>=latex'},
hightlightPath/.style={blue},
path1/.style={blue,line width=1mm},
path2/.style={green,line width=0.4mm}]
\node[vertex,label={[defaultLabel]below left:1}] (1) at (1. , 1.) {}; 
\node[vertex,label={[defaultLabel]below left:2}] (2) at (1. , 2.) {}; 
\node[vertex,label={[defaultLabel]below left:3}] (3) at (1. , 3.) {}; 
\node[vertex,label={[defaultLabel]below left:4}] (4) at (1. , 4.) {}; 
\node[vertex,label={[defaultLabel]below left:5}] (5) at (2. , 1.) {}; 
\node[vertex,label={[defaultLabel]below left:6}] (6) at (2. , 2.) {}; 
\node[vertex,label={[defaultLabel]below left:7}] (7) at (2. , 3.) {}; 
\node[vertex,label={[defaultLabel]below left:8}] (8) at (2. , 4.) {}; 
\node[vertex,label={[defaultLabel]below left:9}] (9) at (3. , 1.) {}; 
\node[vertex,label={[defaultLabel]below left:10}] (10) at (3. , 2.) {}; 
\node[vertex,label={[defaultLabel]below left:11}] (11) at (3. , 3.) {}; 
\node[vertex,label={[defaultLabel]below left:12}] (12) at (3. , 4.) {}; 
\node[vertex,label={[defaultLabel]below left:13}] (13) at (4. , 1.) {}; 
\node[vertex,label={[defaultLabel]below left:14}] (14) at (4. , 2.) {}; 
\node[vertex,label={[defaultLabel]below left:15}] (15) at (4. , 3.) {}; 
\node[vertex,label={[defaultLabel]below left:16}] (16) at (4. , 4.) {}; 
\node[vertex,label={[defaultLabel]below left:17}] (17) at (5. , 1.) {}; 
\node[vertex,label={[defaultLabel]below left:18}] (18) at (5. , 2.) {}; 
\node[vertex,label={[defaultLabel]below left:19}] (19) at (5. , 3.) {}; 
\node[vertex,label={[defaultLabel]below left:20}] (20) at (5. , 4.) {}; 
\draw[edge,](1) -- (2); 
\draw[edge,](1) -- (5); 
\draw[edge,](2) -- (3); 
\draw[edge,](2) -- (6); 
\draw[edge,](3) -- (4); 
\draw[edge,](3) -- (7); 
\draw[edge,](4) -- (8); 
\draw[edge,](5) -- (6); 
\draw[edge,](5) -- (9); 
\draw[edge,](6) -- (7); 
\draw[edge,](6) -- (10); 
\draw[edge,](7) -- (8); 
\draw[edge,](7) -- (11); 
\draw[edge,](8) -- (12); 
\draw[edge,](9) -- (10); 
\draw[edge,](9) -- (13); 
\draw[edge,](10) -- (11); 
\draw[edge,](10) -- (14); 
\draw[edge,](11) -- (12); 
\draw[edge,](11) -- (15); 
\draw[edge,](12) -- (16); 
\draw[edge,](13) -- (14); 
\draw[edge,](13) -- (17); 
\draw[edge,](14) -- (15); 
\draw[edge,](14) -- (18); 
\draw[edge,](15) -- (16); 
\draw[edge,](15) -- (19); 
\draw[edge,](16) -- (20); 
\draw[edge,](17) -- (18); 
\draw[edge,](18) -- (19); 
\draw[edge,](19) -- (20); 

\draw[edge,path1] (1) -- (5);
\draw[edge,path1] (5) -- (6);
\draw[edge,path1] (6) -- (7);
\draw[edge,path1] (7) -- (11);
\draw[edge,path1] (11) -- (15);

\draw[edge,path2] (3) -- (7);
\draw[edge,path2] (7) -- (11);
\draw[edge,path2] (11) -- (15);
\end{tikzpicture}
	\end{center}
	\end{minipage}
	\hfil
	\begin{minipage}{0.45\textwidth}
	\begin{center}
		\begin{tikzpicture}
[vertex/.style={circle,minimum size=1.5mm,inner sep=0pt,fill=black},
edge/.style={thick},
splitVertex/.style={fill=red},
splitLabel/.style={red},
highlightLabel/.style={blue},
defaultLabel/.style={black},
newEdge/.style={red},
changedEdge/.style={red},
leftEnd/.style={<-,>=latex'},
rightEnd/.style={->,>=latex'},
hightlightPath/.style={blue},
path1/.style={blue,line width=1mm},
path2/.style={green,line width=0.4mm}]
\node[vertex,label={[defaultLabel]below left:1}] (1) at (1. , 1.) {}; 
\node[vertex,label={[defaultLabel]below left:2}] (2) at (1. , 2.) {}; 
\node[vertex,label={[defaultLabel]below left:3}] (3) at (1. , 3.) {}; 
\node[vertex,label={[defaultLabel]below left:4}] (4) at (1. , 4.) {}; 
\node[vertex,label={[defaultLabel]below left:5}] (5) at (2. , 1.) {}; 
\node[vertex,label={[defaultLabel]below left:6}] (6) at (2. , 2.) {}; 
\node[vertex,label={[defaultLabel]below left:7}] (7) at (2. , 3.) {}; 
\node[vertex,label={[defaultLabel]below left:8}] (8) at (2. , 4.) {}; 
\node[vertex,label={[defaultLabel]below left:9}] (9) at (3. , 1.) {}; 
\node[vertex,label={[defaultLabel]below left:10}] (10) at (3. , 2.) {}; 
\node[vertex,label={[defaultLabel]below left:11}] (11) at (3. , 3.) {}; 
\node[vertex,label={[defaultLabel]below left:12}] (12) at (3. , 4.) {}; 
\node[vertex,label={[defaultLabel]below left:13}] (13) at (4. , 1.) {}; 
\node[vertex,label={[defaultLabel]below left:14}] (14) at (4. , 2.) {}; 
\node[vertex,label={[defaultLabel]below left:15}] (15) at (4. , 3.) {}; 
\node[vertex,label={[defaultLabel]below left:16}] (16) at (4. , 4.) {}; 
\node[vertex,label={[defaultLabel]below left:17}] (17) at (5. , 1.) {}; 
\node[vertex,label={[defaultLabel]below left:18}] (18) at (5. , 2.) {}; 
\node[vertex,label={[defaultLabel]below left:19}] (19) at (5. , 3.) {}; 
\node[vertex,label={[defaultLabel]below left:20}] (20) at (5. , 4.) {}; 
\draw[edge,](1) -- (2); 
\draw[edge,](1) -- (5); 
\draw[edge,](2) -- (3); 
\draw[edge,](2) -- (6); 
\draw[edge,](3) -- (4); 
\draw[edge,](3) -- (7); 
\draw[edge,](4) -- (8); 
\draw[edge,](5) -- (6); 
\draw[edge,](5) -- (9); 
\draw[edge,](6) -- (7); 
\draw[edge,](6) -- (10); 
\draw[edge,](7) -- (8); 
\draw[edge,](7) -- (11); 
\draw[edge,](8) -- (12); 
\draw[edge,](9) -- (10); 
\draw[edge,](9) -- (13); 
\draw[edge,](10) -- (11); 
\draw[edge,](10) -- (14); 
\draw[edge,](11) -- (12); 
\draw[edge,](11) -- (15); 
\draw[edge,](12) -- (16); 
\draw[edge,](13) -- (14); 
\draw[edge,](13) -- (17); 
\draw[edge,](14) -- (15); 
\draw[edge,](14) -- (18); 
\draw[edge,](15) -- (16); 
\draw[edge,](15) -- (19); 
\draw[edge,](16) -- (20); 
\draw[edge,](17) -- (18); 
\draw[edge,](18) -- (19); 
\draw[edge,](19) -- (20); 

\draw[edge,path1] (1) -- (5);
\draw[edge,path1] (5) -- (6);
\draw[edge,path1] (6) -- (7);
\draw[edge,path1] (6) -- (10);
\draw[edge,path1] (10) -- (14);
\draw[edge,path1] (14) -- (15);

\draw[edge,path2] (3) -- (7);
\draw[edge,path2] (7) -- (6);
\draw[edge,path2] (6) -- (10);
\draw[edge,path2] (10) -- (14);
\draw[edge,path2] (14) -- (15);
\end{tikzpicture}
	\end{center}
	\end{minipage}\hfil
	%row 2
	\\[4mm]
	\begin{minipage}{0.45\textwidth}
	\begin{center}
		\begin{tikzpicture}
[vertex/.style={circle,minimum size=1.5mm,inner sep=0pt,fill=black},
edge/.style={thick},
splitVertex/.style={fill=red},
splitLabel/.style={red},
highlightLabel/.style={blue},
defaultLabel/.style={black},
newEdge/.style={red},
changedEdge/.style={red},
leftEnd/.style={<-,>=latex'},
rightEnd/.style={->,>=latex'},
hightlightPath/.style={blue},
path1/.style={blue,line width=1mm},
path2/.style={green,line width=0.4mm}]
\node[vertex,label={[defaultLabel]below left:1}] (1) at (1. , 1.) {}; 
\node[vertex,label={[defaultLabel]below left:2}] (2) at (1. , 2.) {}; 
\node[vertex,label={[defaultLabel]below left:3}] (3) at (1. , 3.) {}; 
\node[vertex,label={[defaultLabel]below left:4}] (4) at (1. , 4.) {}; 
\node[vertex,label={[defaultLabel]below left:5}] (5) at (2. , 1.) {}; 
\node[vertex,label={[defaultLabel]below left:6}] (6) at (2. , 2.) {}; 
\node[vertex,label={[defaultLabel]below left:7}] (7) at (2. , 3.) {}; 
\node[vertex,label={[defaultLabel]below left:8}] (8) at (2. , 4.) {}; 
\node[vertex,label={[defaultLabel]below left:9}] (9) at (3. , 1.) {}; 
\node[vertex,label={[defaultLabel]below left:10}] (10) at (3. , 2.) {}; 
\node[vertex,label={[defaultLabel]below left:11}] (11) at (3. , 3.) {}; 
\node[vertex,label={[defaultLabel]below left:12}] (12) at (3. , 4.) {}; 
\node[vertex,label={[defaultLabel]below left:13}] (13) at (4. , 1.) {}; 
\node[vertex,label={[defaultLabel]below left:14}] (14) at (4. , 2.) {}; 
\node[vertex,label={[defaultLabel]below left:15}] (15) at (4. , 3.) {}; 
\node[vertex,label={[defaultLabel]below left:16}] (16) at (4. , 4.) {}; 
\node[vertex,label={[defaultLabel]below left:17}] (17) at (5. , 1.) {}; 
\node[vertex,label={[defaultLabel]below left:18}] (18) at (5. , 2.) {}; 
\node[vertex,label={[defaultLabel]below left:19}] (19) at (5. , 3.) {}; 
\node[vertex,label={[defaultLabel]below left:20}] (20) at (5. , 4.) {}; 
\draw[edge,](1) -- (2); 
\draw[edge,](1) -- (5); 
\draw[edge,](2) -- (3); 
\draw[edge,](2) -- (6); 
\draw[edge,](3) -- (4); 
\draw[edge,](3) -- (7); 
\draw[edge,](4) -- (8); 
\draw[edge,](5) -- (6); 
\draw[edge,](5) -- (9); 
\draw[edge,](6) -- (7); 
\draw[edge,](6) -- (10); 
\draw[edge,](7) -- (8); 
\draw[edge,](7) -- (11); 
\draw[edge,](8) -- (12); 
\draw[edge,](9) -- (10); 
\draw[edge,](9) -- (13); 
\draw[edge,](10) -- (11); 
\draw[edge,](10) -- (14); 
\draw[edge,](11) -- (12); 
\draw[edge,](11) -- (15); 
\draw[edge,](12) -- (16); 
\draw[edge,](13) -- (14); 
\draw[edge,](13) -- (17); 
\draw[edge,](14) -- (15); 
\draw[edge,](14) -- (18); 
\draw[edge,](15) -- (16); 
\draw[edge,](15) -- (19); 
\draw[edge,](16) -- (20); 
\draw[edge,](17) -- (18); 
\draw[edge,](18) -- (19); 
\draw[edge,](19) -- (20); 

\draw[edge,path1] (1) -- (5);
\draw[edge,path1] (5) -- (6);
\draw[edge,path1] (6) -- (10);
\draw[edge,path1] (10) -- (14);
\draw[edge,path1] (14) -- (15);

\draw[edge,path2] (3) -- (7);
\draw[edge,path2] (7) -- (6);
\draw[edge,path2] (6) -- (10);
\draw[edge,path2] (10) -- (14);
\draw[edge,path2] (14) -- (15);
\end{tikzpicture}
	\end{center}
	\end{minipage}
	\hfil
	\begin{minipage}{0.45\textwidth}
	\begin{center}
		\begin{tikzpicture}
[vertex/.style={circle,minimum size=1.5mm,inner sep=0pt,fill=black},
edge/.style={thick},
splitVertex/.style={fill=red},
splitLabel/.style={red},
highlightLabel/.style={blue},
defaultLabel/.style={black},
newEdge/.style={red},
changedEdge/.style={red},
leftEnd/.style={<-,>=latex'},
rightEnd/.style={->,>=latex'},
hightlightPath/.style={blue},
path1/.style={blue,line width=1mm},
path2/.style={green,line width=0.4mm}]
\node[vertex,label={[defaultLabel]below left:1}] (1) at (1. , 1.) {}; 
\node[vertex,label={[defaultLabel]below left:2}] (2) at (1. , 2.) {}; 
\node[vertex,label={[defaultLabel]below left:3}] (3) at (1. , 3.) {}; 
\node[vertex,label={[defaultLabel]below left:4}] (4) at (1. , 4.) {}; 
\node[vertex,label={[defaultLabel]below left:5}] (5) at (2. , 1.) {}; 
\node[vertex,label={[defaultLabel]below left:6}] (6) at (2. , 2.) {}; 
\node[vertex,label={[defaultLabel]below left:7}] (7) at (2. , 3.) {}; 
\node[vertex,label={[defaultLabel]below left:8}] (8) at (2. , 4.) {}; 
\node[vertex,label={[defaultLabel]below left:9}] (9) at (3. , 1.) {}; 
\node[vertex,label={[defaultLabel]below left:10}] (10) at (3. , 2.) {}; 
\node[vertex,label={[defaultLabel]below left:11}] (11) at (3. , 3.) {}; 
\node[vertex,label={[defaultLabel]below left:12}] (12) at (3. , 4.) {}; 
\node[vertex,label={[defaultLabel]below left:13}] (13) at (4. , 1.) {}; 
\node[vertex,label={[defaultLabel]below left:14}] (14) at (4. , 2.) {}; 
\node[vertex,label={[defaultLabel]below left:15}] (15) at (4. , 3.) {}; 
\node[vertex,label={[defaultLabel]below left:16}] (16) at (4. , 4.) {}; 
\node[vertex,label={[defaultLabel]below left:17}] (17) at (5. , 1.) {}; 
\node[vertex,label={[defaultLabel]below left:18}] (18) at (5. , 2.) {}; 
\node[vertex,label={[defaultLabel]below left:19}] (19) at (5. , 3.) {}; 
\node[vertex,label={[defaultLabel]below left:20}] (20) at (5. , 4.) {}; 
\draw[edge,](1) -- (2); 
\draw[edge,](1) -- (5); 
\draw[edge,](2) -- (3); 
\draw[edge,](2) -- (6); 
\draw[edge,](3) -- (4); 
\draw[edge,](3) -- (7); 
\draw[edge,](4) -- (8); 
\draw[edge,](5) -- (6); 
\draw[edge,](5) -- (9); 
\draw[edge,](6) -- (7); 
\draw[edge,](6) -- (10); 
\draw[edge,](7) -- (8); 
\draw[edge,](7) -- (11); 
\draw[edge,](8) -- (12); 
\draw[edge,](9) -- (10); 
\draw[edge,](9) -- (13); 
\draw[edge,](10) -- (11); 
\draw[edge,](10) -- (14); 
\draw[edge,](11) -- (12); 
\draw[edge,](11) -- (15); 
\draw[edge,](12) -- (16); 
\draw[edge,](13) -- (14); 
\draw[edge,](13) -- (17); 
\draw[edge,](14) -- (15); 
\draw[edge,](14) -- (18); 
\draw[edge,](15) -- (16); 
\draw[edge,](15) -- (19); 
\draw[edge,](16) -- (20); 
\draw[edge,](17) -- (18); 
\draw[edge,](18) -- (19); 
\draw[edge,](19) -- (20); 

\draw[edge,path1] (1) -- (5);
\draw[edge,path1] (5) -- (6);
\draw[edge,path1] (6) -- (10);
\draw[edge,path1] (10) -- (14);
\draw[edge,path1] (14) -- (15);

\draw[edge,path2] (3) -- (2);
\draw[edge,path2] (2) -- (1);
\draw[edge,path2] (1) -- (5);
\draw[edge,path2] (5) -- (6);
\draw[edge,path2] (6) -- (10);
\draw[edge,path2] (10) -- (14);
\draw[edge,path2] (14) -- (15);
\end{tikzpicture}
	\end{center}
	\end{minipage}\hfil
	%end
	\caption{An illustration of applying one round of \textsc{ImproveSharedSubpath} (Algorithm~\ref{algo:improvePath}). Top left: a graph in which the two paths $P_i$ and $P_j$ shown in blue and green share the subpath (7,11,15). Top right: after calculating a shortest path from 7 to 15 for the combined weight (lines 3--8 in Algorithm~\ref{algo:improvePath}), the blue path contains the circle (6,7,6). Bottom left: this circle gets removed from the blue path. Bottom right:
	an updated shortest green path results in a new shared subpath that
	could be further improved by applying \textsc{ImproveSharedSubpath} recursively.}
	\label{fig:recursiveImprove}
\end{figure}
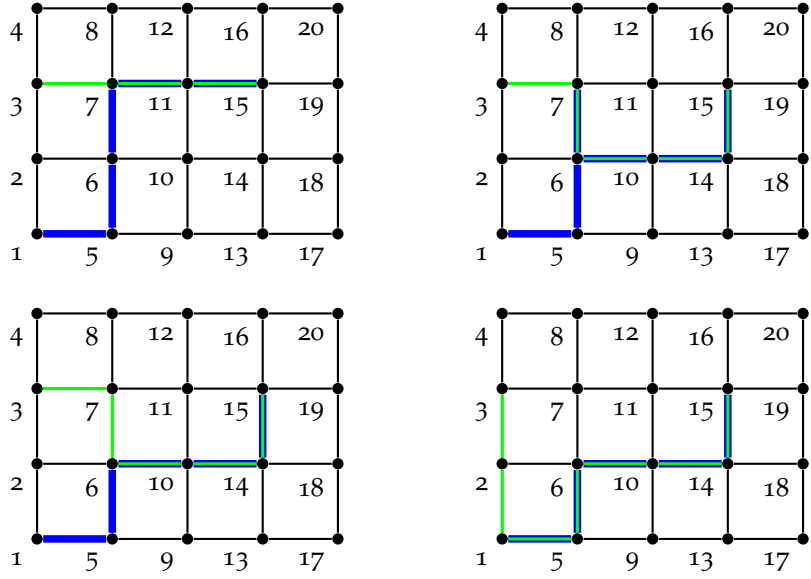

\subsection{Optimized Lensing Deformations}\label{subsec:algoLensing}
\begin{algorithm}[ht]
	\caption{Optimized Lensing Deformation}\label{algo:lensing}
	\begin{algorithmic}[1]
		\Procedure{LensingDeformation}{$G,\Gamma$}
		\State select $\Gamma_j$ with highest weight
		\For{ $\mathcal{D} \in \{LDU,LU,\dots\} $}
		\State $G^\mathcal{D} = (G_1,\dots,G_{j-1},\text{decomposition }\mathcal{D} \text{ of } G_j,\dots,G_{j+1},\dots,G_n)$
		\State $\Gamma^\mathcal{D} = (\Gamma_1,\dots,\Gamma_{j-1},\underbrace{\Gamma_j,\Gamma_j,\Gamma_j,\dots,\Gamma_j,\Gamma_j,\Gamma_j}_{\text{\# copies = \# factors in } \mathcal{D}},\Gamma_{j+1},\dots,\Gamma_n)$
			\State $(\tilde G^\mathcal{D},\tilde\Gamma^\mathcal{D}) = \text{SimpleDeformation}(G^\mathcal{D},\Gamma^\mathcal{D})$
		\EndFor
		\State return $(\tilde G^\mathcal{D},\tilde\Gamma^\mathcal{D})$ with lowest weight
		\EndProcedure
	\end{algorithmic}
\end{algorithm}

A single step of the optimized lensing deformation (Algorithm \ref{algo:lensing}) aims at improving the
dominant part of the contour by trying various decompositions (factorizations) of its jump matrix to which, then, the optimized simple deformation (Algorithm \ref{algo:optPath}) 
is applied. The contour parts which originate from such a lensing deformation (e.g. $L$, $D$ and $U$ in Fig.~\ref{fig:lensing}) have, however, to satisfy an additional constraint, namely condition (iv) of §\ref{sec:deform}: their spatial
order has to be preserved. To calculate shortest paths (line \ref{line:shortestpaths} in Algorithm~\ref{algo:optPath}) subject to this additional condition, we distinguish between the following three cases (an illustration can be found in Fig.~\ref{fig:algoLensing}): 

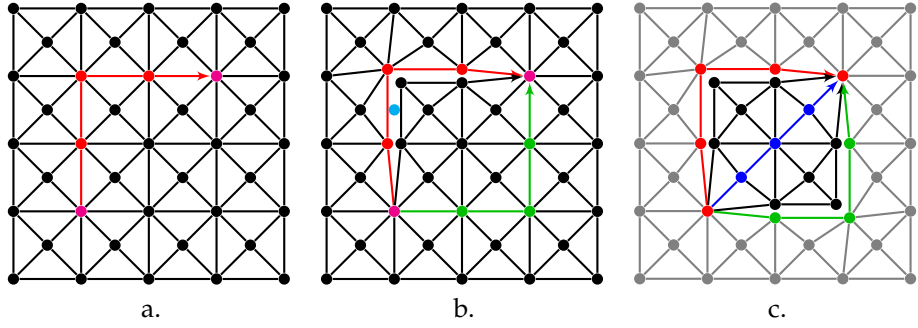
\begin{figure}[h]
	\begin{minipage}{0.30\textwidth}
		\begin{center}
			\begin{tikzpicture}
[vertex/.style={circle,minimum size=1.5mm,inner sep=0pt,fill=black},
edge/.style={thick},
splitVertex/.style={fill=red},
splitLabel/.style={red},
highlightLabel/.style={blue},
defaultLabel/.style={black},
newEdge/.style={red},
changedEdge/.style={red},
end/.style={->,>=latex'},
leftEnd/.style={<-,>=latex'},
rightEnd/.style={->,>=latex'},
highlightEdge/.style={blue},
highlightVertex/.style={blue},
scale=0.9]\node[vertex,] (1) at (1. , 1.) {}; 
\node[vertex,] (2) at (1. , 2.) {}; 
\node[vertex,] (3) at (1. , 3.) {}; 
\node[vertex,] (4) at (1. , 4.) {}; 
\node[vertex,] (5) at (1. , 5.) {}; 
\node[vertex,] (6) at (2. , 1.) {}; 
\node[vertex,magenta] (7) at (2. , 2.) {}; 
\node[vertex,red] (8) at (2. , 3.) {}; 
\node[vertex,red] (9) at (2. , 4.) {}; 
\node[vertex,] (10) at (2. , 5.) {}; 
\node[vertex,] (11) at (3. , 1.) {}; 
\node[vertex,] (12) at (3. , 2.) {}; 
\node[vertex,] (13) at (3. , 3.) {}; 
\node[vertex,red] (14) at (3. , 4.) {}; 
\node[vertex,] (15) at (3. , 5.) {}; 
\node[vertex,] (16) at (4. , 1.) {}; 
\node[vertex,] (17) at (4. , 2.) {}; 
\node[vertex,] (18) at (4. , 3.) {}; 
\node[vertex,magenta] (19) at (4. , 4.) {}; 
\node[vertex,] (20) at (4. , 5.) {}; 
\node[vertex,] (21) at (5. , 1.) {}; 
\node[vertex,] (22) at (5. , 2.) {}; 
\node[vertex,] (23) at (5. , 3.) {}; 
\node[vertex,] (24) at (5. , 4.) {}; 
\node[vertex,] (25) at (5. , 5.) {}; 
\node[vertex,] (26) at (1.5 , 1.5) {}; 
\node[vertex,] (27) at (1.5 , 2.5) {}; 
\node[vertex,] (28) at (1.5 , 3.5) {}; 
\node[vertex,] (29) at (1.5 , 4.5) {}; 
\node[vertex,] (30) at (2.5 , 1.5) {}; 
\node[vertex,] (31) at (2.5 , 2.5) {}; 
\node[vertex,] (32) at (2.5 , 3.5) {}; 
\node[vertex,] (33) at (2.5 , 4.5) {}; 
\node[vertex,] (34) at (3.5 , 1.5) {}; 
\node[vertex,] (35) at (3.5 , 2.5) {}; 
\node[vertex,] (36) at (3.5 , 3.5) {}; 
\node[vertex,] (37) at (3.5 , 4.5) {}; 
\node[vertex,] (38) at (4.5 , 1.5) {}; 
\node[vertex,] (39) at (4.5 , 2.5) {}; 
\node[vertex,] (40) at (4.5 , 3.5) {}; 
\node[vertex,] (41) at (4.5 , 4.5) {}; 
\draw[edge,](1) -- (2) ; 
\draw[edge,](1) -- (6) ; 
\draw[edge,](1) -- (26) ; 
\draw[edge,](2) -- (3) ; 
\draw[edge,](2) -- (7) ; 
\draw[edge,](2) -- (26) ; 
\draw[edge,](2) -- (27) ; 
\draw[edge,](3) -- (4) ; 
\draw[edge,](3) -- (8) ; 
\draw[edge,](3) -- (27) ; 
\draw[edge,](3) -- (28) ; 
\draw[edge,](4) -- (5) ; 
\draw[edge,](4) -- (9) ; 
\draw[edge,](4) -- (28) ; 
\draw[edge,](4) -- (29) ; 
\draw[edge,](5) -- (10) ; 
\draw[edge,](5) -- (29) ; 
\draw[edge,](6) -- (7) ; 
\draw[edge,](6) -- (11) ; 
\draw[edge,](6) -- (26) ; 
\draw[edge,](6) -- (30) ; 
\draw[edge,red](7) -- (8) ; 
\draw[edge,](7) -- (12) ; 
\draw[edge,](7) -- (26) ; 
\draw[edge,](7) -- (27) ; 
\draw[edge,](7) -- (30) ; 
\draw[edge,](7) -- (31) ; 
\draw[edge,red](8) -- (9) ; 
\draw[edge,](8) -- (13) ; 
\draw[edge,](8) -- (27) ; 
\draw[edge,](8) -- (28) ; 
\draw[edge,](8) -- (31) ; 
\draw[edge,](8) -- (32) ; 
\draw[edge,](9) -- (10) ; 
\draw[edge,red](9) -- (14) ; 
\draw[edge,](9) -- (28) ; 
\draw[edge,](9) -- (29) ; 
\draw[edge,](9) -- (32) ; 
\draw[edge,](9) -- (33) ; 
\draw[edge,](10) -- (15) ; 
\draw[edge,](10) -- (29) ; 
\draw[edge,](10) -- (33) ; 
\draw[edge,](11) -- (12) ; 
\draw[edge,](11) -- (16) ; 
\draw[edge,](11) -- (30) ; 
\draw[edge,](11) -- (34) ; 
\draw[edge,](12) -- (13) ; 
\draw[edge,](12) -- (17) ; 
\draw[edge,](12) -- (30) ; 
\draw[edge,](12) -- (31) ; 
\draw[edge,](12) -- (34) ; 
\draw[edge,](12) -- (35) ; 
\draw[edge,](13) -- (14) ; 
\draw[edge,](13) -- (18) ; 
\draw[edge,](13) -- (31) ; 
\draw[edge,](13) -- (32) ; 
\draw[edge,](13) -- (35) ; 
\draw[edge,](13) -- (36) ; 
\draw[edge,](14) -- (15) ; 
\draw[edge,rightEnd,red](14) -- (19) ; 
\draw[edge,](14) -- (32) ; 
\draw[edge,](14) -- (33) ; 
\draw[edge,](14) -- (36) ; 
\draw[edge,](14) -- (37) ; 
\draw[edge,](15) -- (20) ; 
\draw[edge,](15) -- (33) ; 
\draw[edge,](15) -- (37) ; 
\draw[edge,](16) -- (17) ; 
\draw[edge,](16) -- (21) ; 
\draw[edge,](16) -- (34) ; 
\draw[edge,](16) -- (38) ; 
\draw[edge,](17) -- (18) ; 
\draw[edge,](17) -- (22) ; 
\draw[edge,](17) -- (34) ; 
\draw[edge,](17) -- (35) ; 
\draw[edge,](17) -- (38) ; 
\draw[edge,](17) -- (39) ; 
\draw[edge,](18) -- (19) ; 
\draw[edge,](18) -- (23) ; 
\draw[edge,](18) -- (35) ; 
\draw[edge,](18) -- (36) ; 
\draw[edge,](18) -- (39) ; 
\draw[edge,](18) -- (40) ; 
\draw[edge,](19) -- (20) ; 
\draw[edge,](19) -- (24) ; 
\draw[edge,](19) -- (36) ; 
\draw[edge,](19) -- (37) ; 
\draw[edge,](19) -- (40) ; 
\draw[edge,](19) -- (41) ; 
\draw[edge,](20) -- (25) ; 
\draw[edge,](20) -- (37) ; 
\draw[edge,](20) -- (41) ; 
\draw[edge,](21) -- (22) ; 
\draw[edge,](21) -- (38) ; 
\draw[edge,](22) -- (23) ; 
\draw[edge,](22) -- (38) ; 
\draw[edge,](22) -- (39) ; 
\draw[edge,](23) -- (24) ; 
\draw[edge,](23) -- (39) ; 
\draw[edge,](23) -- (40) ; 
\draw[edge,](24) -- (25) ; 
\draw[edge,](24) -- (40) ; 
\draw[edge,](24) -- (41) ; 
\draw[edge,](25) -- (41) ; 
\end{tikzpicture}
			a.
		\end{center}
	\end{minipage}
	\hfil
	\begin{minipage}{0.30\textwidth}
		\begin{center}
			\begin{tikzpicture}
[vertex/.style={circle,minimum size=1.5mm,inner sep=0pt,fill=black},
edge/.style={thick},
splitVertex/.style={},
splitLabel/.style={red},
highlightLabel/.style={blue},
defaultLabel/.style={black},
newEdge/.style={},
changedEdge/.style={},
end/.style={->,>=latex'},
leftEnd/.style={<-,>=latex'},
rightEnd/.style={->,>=latex'},
highlightEdge/.style={blue},
highlightVertex/.style={blue},
scale=0.9]\node[vertex,] (1) at (1. , 1.) {}; 
\node[vertex,] (2) at (1. , 2.) {}; 
\node[vertex,] (3) at (1. , 3.) {}; 
\node[vertex,] (4) at (1. , 4.) {}; 
\node[vertex,] (5) at (1. , 5.) {}; 
\node[vertex,] (6) at (2. , 1.) {}; 
\node[vertex,magenta] (7) at (2. , 2.) {}; 
\node[vertex,red] (8) at (1.9 , 3.) {}; 
\node[vertex,red] (9) at (1.9 , 4.1) {}; 
\node[vertex,] (10) at (2. , 5.) {}; 
\node[vertex,] (11) at (3. , 1.) {}; 
\node[vertex,green!75!black] (12) at (3. , 2.) {}; 
\node[vertex,] (13) at (3. , 3.) {}; 
\node[vertex,red] (14) at (3. , 4.1) {}; 
\node[vertex,] (15) at (3. , 5.) {}; 
\node[vertex,] (16) at (4. , 1.) {}; 
\node[vertex,green!75!black] (17) at (4. , 2.) {}; 
\node[vertex,green!75!black] (18) at (4. , 3.) {}; 
\node[vertex,magenta] (19) at (4. , 4.) {}; 
\node[vertex,] (20) at (4. , 5.) {}; 
\node[vertex,] (21) at (5. , 1.) {}; 
\node[vertex,] (22) at (5. , 2.) {}; 
\node[vertex,] (23) at (5. , 3.) {}; 
\node[vertex,] (24) at (5. , 4.) {}; 
\node[vertex,] (25) at (5. , 5.) {}; 
\node[vertex,] (26) at (1.5 , 1.5) {}; 
\node[vertex,] (27) at (1.5 , 2.5) {}; 
\node[vertex,] (28) at (1.5 , 3.5) {}; 
\node[vertex,] (29) at (1.5 , 4.5) {}; 
\node[vertex,] (30) at (2.5 , 1.5) {}; 
\node[vertex,] (31) at (2.5 , 2.5) {}; 
\node[vertex,] (32) at (2.5 , 3.5) {}; 
\node[vertex,] (33) at (2.5 , 4.5) {}; 
\node[vertex,] (34) at (3.5 , 1.5) {}; 
\node[vertex,] (35) at (3.5 , 2.5) {}; 
\node[vertex,] (36) at (3.5 , 3.5) {}; 
\node[vertex,] (37) at (3.5 , 4.5) {}; 
\node[vertex,] (38) at (4.5 , 1.5) {}; 
\node[vertex,] (39) at (4.5 , 2.5) {}; 
\node[vertex,] (40) at (4.5 , 3.5) {}; 
\node[vertex,] (41) at (4.5 , 4.5) {}; 
\node[vertex,splitVertex,black] (42) at (2.1 , 3.) {}; 
\node[vertex,splitVertex,black] (43) at (2.1 , 3.9) {}; 
\node[vertex,splitVertex,black] (44) at (3. , 3.9) {}; 
\draw[edge,](1) -- (2) ; 
\draw[edge,](1) -- (6) ; 
\draw[edge,](1) -- (26) ; 
\draw[edge,](2) -- (3) ; 
\draw[edge,](2) -- (7) ; 
\draw[edge,](2) -- (26) ; 
\draw[edge,](2) -- (27) ; 
\draw[edge,](3) -- (4) ; 
\draw[edge,](3) -- (8) ; 
\draw[edge,](3) -- (27) ; 
\draw[edge,](3) -- (28) ; 
\draw[edge,](4) -- (5) ; 
\draw[edge,](4) -- (9) ; 
\draw[edge,](4) -- (28) ; 
\draw[edge,](4) -- (29) ; 
\draw[edge,](5) -- (10) ; 
\draw[edge,](5) -- (29) ; 
\draw[edge,](6) -- (7) ; 
\draw[edge,](6) -- (11) ; 
\draw[edge,](6) -- (26) ; 
\draw[edge,](6) -- (30) ; 
\draw[edge,red](7) -- (8) ; 
\draw[edge,green!75!black](7) -- (12) ; 
\draw[edge,](7) -- (26) ; 
\draw[edge,](7) -- (27) ; 
\draw[edge,](7) -- (30) ; 
\draw[edge,](7) -- (31) ; 
\draw[edge,changedEdge,black](7) -- (42) ; 
\draw[edge,red](8) -- (9) ; 
\draw[edge,](8) -- (27) ; 
\draw[edge,](8) -- (28) ; 
\draw[edge,](9) -- (10) ; 
\draw[edge,red](9) -- (14) ; 
\draw[edge,](9) -- (28) ; 
\draw[edge,](9) -- (29) ; 
\draw[edge,](9) -- (33) ; 
\draw[edge,](10) -- (15) ; 
\draw[edge,](10) -- (29) ; 
\draw[edge,](10) -- (33) ; 
\draw[edge,](11) -- (12) ; 
\draw[edge,](11) -- (16) ; 
\draw[edge,](11) -- (30) ; 
\draw[edge,](11) -- (34) ; 
\draw[edge,](12) -- (13) ; 
\draw[edge,green!75!black](12) -- (17) ; 
\draw[edge,](12) -- (30) ; 
\draw[edge,](12) -- (31) ; 
\draw[edge,](12) -- (34) ; 
\draw[edge,](12) -- (35) ; 
\draw[edge,](13) -- (18) ; 
\draw[edge,](13) -- (31) ; 
\draw[edge,](13) -- (32) ; 
\draw[edge,](13) -- (35) ; 
\draw[edge,](13) -- (36) ; 
\draw[edge,changedEdge,](13) -- (42) ; 
\draw[edge,changedEdge,](13) -- (44) ; 
\draw[edge,](14) -- (15) ; 
\draw[edge,rightEnd,red](14) -- (19) ; 
\draw[edge,](14) -- (33) ; 
\draw[edge,](14) -- (37) ; 
\draw[edge,](15) -- (20) ; 
\draw[edge,](15) -- (33) ; 
\draw[edge,](15) -- (37) ; 
\draw[edge,](16) -- (17) ; 
\draw[edge,](16) -- (21) ; 
\draw[edge,](16) -- (34) ; 
\draw[edge,](16) -- (38) ; 
\draw[edge,green!75!black](17) -- (18) ; 
\draw[edge,](17) -- (22) ; 
\draw[edge,](17) -- (34) ; 
\draw[edge,](17) -- (35) ; 
\draw[edge,](17) -- (38) ; 
\draw[edge,](17) -- (39) ; 
\draw[edge,rightEnd,green!75!black](18) -- (19) ; 
\draw[edge,](18) -- (23) ; 
\draw[edge,](18) -- (35) ; 
\draw[edge,](18) -- (36) ; 
\draw[edge,](18) -- (39) ; 
\draw[edge,](18) -- (40) ; 
\draw[edge,](19) -- (20) ; 
\draw[edge,](19) -- (24) ; 
\draw[edge,](19) -- (36) ; 
\draw[edge,](19) -- (37) ; 
\draw[edge,](19) -- (40) ; 
\draw[edge,](19) -- (41) ; 
\draw[edge,changedEdge,leftEnd,black](19) -- (44) ; 
\draw[edge,](20) -- (25) ; 
\draw[edge,](20) -- (37) ; 
\draw[edge,](20) -- (41) ; 
\draw[edge,](21) -- (22) ; 
\draw[edge,](21) -- (38) ; 
\draw[edge,](22) -- (23) ; 
\draw[edge,](22) -- (38) ; 
\draw[edge,](22) -- (39) ; 
\draw[edge,](23) -- (24) ; 
\draw[edge,](23) -- (39) ; 
\draw[edge,](23) -- (40) ; 
\draw[edge,](24) -- (25) ; 
\draw[edge,](24) -- (40) ; 
\draw[edge,](24) -- (41) ; 
\draw[edge,](25) -- (41) ; 
\draw[edge,changedEdge,](31) -- (42) ; 
\draw[edge,changedEdge,](32) -- (42) ; 
\draw[edge,changedEdge,](32) -- (43) ; 
\draw[edge,changedEdge,](32) -- (44) ; 
\draw[edge,changedEdge,](36) -- (44) ; 
\draw[edge,newEdge,black](42) -- (43) ; 
\draw[edge,newEdge,black](43) -- (44) ; 
\node[vertex,cyan] (*) at (2,3.5) {}; \end{tikzpicture}
			b.
		\end{center}
	\end{minipage}
	\hfil
	\begin{minipage}{0.30\textwidth}
		\begin{center}
			\begin{tikzpicture}
[vertex/.style={circle,minimum size=1.5mm,inner sep=0pt,fill=black},
edge/.style={thick},
splitVertex/.style={},
splitLabel/.style={red},
highlightLabel/.style={blue},
defaultLabel/.style={black},
newEdge/.style={},
changedEdge/.style={},
end/.style={->,>=latex'},
leftEnd/.style={<-,>=latex'},
rightEnd/.style={->,>=latex'},
highlightEdge/.style={blue},
highlightVertex/.style={blue},
scale=0.9]\node[vertex,gray] (1) at (1. , 1.) {}; 
\node[vertex,gray] (2) at (1. , 2.) {}; 
\node[vertex,gray] (3) at (1. , 3.) {}; 
\node[vertex,gray] (4) at (1. , 4.) {}; 
\node[vertex,gray] (5) at (1. , 5.) {}; 
\node[vertex,gray] (6) at (2. , 1.) {}; 
\node[vertex,red] (7) at (2. , 2.) {}; 
\node[vertex,red] (8) at (1.9 , 3.) {}; 
\node[vertex,red] (9) at (1.9 , 4.1) {}; 
\node[vertex,gray] (10) at (2. , 5.) {}; 
\node[vertex,gray] (11) at (3. , 1.) {}; 
\node[vertex,black] (12) at (3. , 2.1) {}; 
\node[vertex,blue] (13) at (3. , 3.) {}; 
\node[vertex,red] (14) at (3. , 4.1) {}; 
\node[vertex,gray] (15) at (3. , 5.) {}; 
\node[vertex,gray] (16) at (4. , 1.) {}; 
\node[vertex,black] (17) at (3.9 , 2.1) {}; 
\node[vertex,black] (18) at (3.9 , 3.) {}; 
\node[vertex,red] (19) at (4. , 4.) {}; 
\node[vertex,gray] (20) at (4. , 5.) {}; 
\node[vertex,gray] (21) at (5. , 1.) {}; 
\node[vertex,gray] (22) at (5. , 2.) {}; 
\node[vertex,gray] (23) at (5. , 3.) {}; 
\node[vertex,gray] (24) at (5. , 4.) {}; 
\node[vertex,gray] (25) at (5. , 5.) {}; 
\node[vertex,gray] (26) at (1.5 , 1.5) {}; 
\node[vertex,gray] (27) at (1.5 , 2.5) {}; 
\node[vertex,gray] (28) at (1.5 , 3.5) {}; 
\node[vertex,gray] (29) at (1.5 , 4.5) {}; 
\node[vertex,gray] (30) at (2.5 , 1.5) {}; 
\node[vertex,blue] (31) at (2.5 , 2.5) {}; 
\node[vertex,] (32) at (2.5 , 3.5) {}; 
\node[vertex,gray] (33) at (2.5 , 4.5) {}; 
\node[vertex,gray] (34) at (3.5 , 1.5) {}; 
\node[vertex,] (35) at (3.5 , 2.5) {}; 
\node[vertex,blue] (36) at (3.5 , 3.5) {}; 
\node[vertex,gray] (37) at (3.5 , 4.5) {}; 
\node[vertex,gray] (38) at (4.5 , 1.5) {}; 
\node[vertex,gray] (39) at (4.5 , 2.5) {}; 
\node[vertex,gray] (40) at (4.5 , 3.5) {}; 
\node[vertex,gray] (41) at (4.5 , 4.5) {}; 
\node[vertex,splitVertex,black] (42) at (2.1 , 3.) {}; 
\node[vertex,splitVertex,black] (43) at (2.1 , 3.9) {}; 
\node[vertex,splitVertex,black] (44) at (3. , 3.9) {}; 
\node[vertex,splitVertex,green!75!black] (45) at (3. , 1.9) {}; 
\node[vertex,splitVertex,green!75!black] (46) at (4.1 , 1.9) {}; 
\node[vertex,splitVertex,green!75!black] (47) at (4.1 , 3.) {}; 
\draw[edge,gray](1) -- (2) ; 
\draw[edge,gray](1) -- (6) ; 
\draw[edge,gray](1) -- (26) ; 
\draw[edge,gray](2) -- (3) ; 
\draw[edge,gray](2) -- (7) ; 
\draw[edge,gray](2) -- (26) ; 
\draw[edge,gray](2) -- (27) ; 
\draw[edge,gray](3) -- (4) ; 
\draw[edge,gray](3) -- (8) ; 
\draw[edge,gray](3) -- (27) ; 
\draw[edge,gray](3) -- (28) ; 
\draw[edge,gray](4) -- (5) ; 
\draw[edge,gray](4) -- (9) ; 
\draw[edge,gray](4) -- (28) ; 
\draw[edge,gray](4) -- (29) ; 
\draw[edge,gray](5) -- (10) ; 
\draw[edge,gray](5) -- (29) ; 
\draw[edge,gray](6) -- (7) ; 
\draw[edge,gray](6) -- (11) ; 
\draw[edge,gray](6) -- (26) ; 
\draw[edge,gray](6) -- (30) ; 
\draw[edge,red](7) -- (8) ; 
\draw[edge,black](7) -- (12) ; 
\draw[edge,gray](7) -- (26) ; 
\draw[edge,gray](7) -- (27) ; 
\draw[edge,gray](7) -- (30) ; 
\draw[edge,blue](7) -- (31) ; 
\draw[edge,changedEdge,black](7) -- (42) ; 
\draw[edge,changedEdge,green!75!black](7) -- (45) ; 
\draw[edge,red](8) -- (9) ; 
\draw[edge,gray](8) -- (27) ; 
\draw[edge,gray](8) -- (28) ; 
\draw[edge,gray](9) -- (10) ; 
\draw[edge,red](9) -- (14) ; 
\draw[edge,gray](9) -- (28) ; 
\draw[edge,gray](9) -- (29) ; 
\draw[edge,gray](9) -- (33) ; 
\draw[edge,gray](10) -- (15) ; 
\draw[edge,gray](10) -- (29) ; 
\draw[edge,gray](10) -- (33) ; 
\draw[edge,gray](11) -- (16) ; 
\draw[edge,gray](11) -- (30) ; 
\draw[edge,gray](11) -- (34) ; 
\draw[edge,changedEdge,gray](11) -- (45) ; 
\draw[edge,](12) -- (13) ; 
\draw[edge,black](12) -- (17) ; 
\draw[edge,](12) -- (31) ; 
\draw[edge,](12) -- (35) ; 
\draw[edge,](13) -- (18) ; 
\draw[edge,blue](13) -- (31) ; 
\draw[edge,](13) -- (32) ; 
\draw[edge,](13) -- (35) ; 
\draw[edge,blue](13) -- (36) ; 
\draw[edge,changedEdge,](13) -- (42) ; 
\draw[edge,changedEdge,](13) -- (44) ; 
\draw[edge,gray](14) -- (15) ; 
\draw[edge,rightEnd,red](14) -- (19) ; 
\draw[edge,gray](14) -- (33) ; 
\draw[edge,gray](14) -- (37) ; 
\draw[edge,gray](15) -- (20) ; 
\draw[edge,gray](15) -- (33) ; 
\draw[edge,gray](15) -- (37) ; 
\draw[edge,gray](16) -- (21) ; 
\draw[edge,gray](16) -- (34) ; 
\draw[edge,gray](16) -- (38) ; 
\draw[edge,changedEdge,gray](16) -- (46) ; 
\draw[edge](17) -- (18) ; 
\draw[edge,](17) -- (35) ; 
\draw[edge,rightEnd,black](18) -- (19) ; 
\draw[edge,](18) -- (35) ; 
\draw[edge,](18) -- (36) ; 
\draw[edge,gray](19) -- (20) ; 
\draw[edge,gray](19) -- (24) ; 
\draw[edge,leftEnd,blue](19) -- (36) ; 
\draw[edge,gray](19) -- (37) ; 
\draw[edge,gray](19) -- (40) ; 
\draw[edge,gray](19) -- (41) ; 
\draw[edge,changedEdge,leftEnd,black](19) -- (44) ; 
\draw[edge,changedEdge,leftEnd,green!75!black](19) -- (47) ; 
\draw[edge,gray](20) -- (25) ; 
\draw[edge,gray](20) -- (37) ; 
\draw[edge,gray](20) -- (41) ; 
\draw[edge,gray](21) -- (22) ; 
\draw[edge,gray](21) -- (38) ; 
\draw[edge,gray](22) -- (23) ; 
\draw[edge,gray](22) -- (38) ; 
\draw[edge,gray](22) -- (39) ; 
\draw[edge,changedEdge,gray](22) -- (46) ; 
\draw[edge,gray](23) -- (24) ; 
\draw[edge,gray](23) -- (39) ; 
\draw[edge,gray](23) -- (40) ; 
\draw[edge,changedEdge,gray](23) -- (47) ; 
\draw[edge,gray](24) -- (25) ; 
\draw[edge,gray](24) -- (40) ; 
\draw[edge,gray](24) -- (41) ; 
\draw[edge,gray](25) -- (41) ; 
\draw[edge,changedEdge,gray](30) -- (45) ; 
\draw[edge,changedEdge,](31) -- (42) ; 
\draw[edge,changedEdge,](32) -- (42) ; 
\draw[edge,changedEdge,](32) -- (43) ; 
\draw[edge,changedEdge,](32) -- (44) ; 
\draw[edge,changedEdge,gray](34) -- (45) ; 
\draw[edge,changedEdge,gray](34) -- (46) ; 
\draw[edge,changedEdge,](36) -- (44) ; 
\draw[edge,changedEdge,gray](38) -- (46) ; 
\draw[edge,changedEdge,gray](39) -- (46) ; 
\draw[edge,changedEdge,gray](39) -- (47) ; 
\draw[edge,changedEdge,gray](40) -- (47) ; 
\draw[edge,newEdge,black](42) -- (43) ; 
\draw[edge,newEdge,black](43) -- (44) ; 
\draw[edge,newEdge,green!75!black](45) -- (46) ; 
\draw[edge,newEdge,green!75!black](46) -- (47) ; 
\end{tikzpicture}
			c.
		\end{center}
	\end{minipage}
	\caption{Illustration of the shortest path calculations done for a lensing deformation as in Fig.~\ref{fig:lensing}.
	Here, all three cases appear: 
	\textbf{a.} All parts of the decomposition are still unprocessed (case 1). A shortest path $p_i$ (red) has to be calculated between the endpoints (magenta) of $\Gamma_i$.
	\textbf{b.} A path $P_l$ for the contour part $\Gamma_l$ (red) to the left of $\Gamma_i$ has already been fixed and the graph has been split accordingly (case 2). Now, $p_i$ must be constructed as the shortest path subject to the constraint that the circle $p_i \join P_l$ {\em encloses} some point $c$ in the split (shown in cyan). All edges are undirected, the arrows just indicate the directions of the paths. 
	\textbf{c.}  $P_l$ (red) to the left side of $p_i$ (blue) and $P_r$ (green) to the right side of $p_i$ have already been fixed and the graph split accordingly (case 3). All vertices that are removed from the graph before calculating the
	shortest path $p_i$ are shown in gray.}\label{fig:algoLensing}
%	\caption{Example of shortest path calculations done for the lensing deformation shown in figure \todo[inline]{add reference to figure} \newline \textbf{Step~1:} No part of the lensing deformation has been fixed, so each part is treated like an ordinary contour part by calculating the shortest path between its endpoints (mangeta). The shortest path with the highest weight is fixed, which is in this case the path for the ``R'' part (red).\newline \textbf{Step~2:} Fixing the path for ``R'' introduces a split into the graph and the additional constraint that ``D'' and ``L'' have to be right of ``R''. This restriction is met by calculating the shortest enclosing arcs between the endpoints (magenta), so that the left path for ``R'' and path for ``L''/``D'' enclose a point (cyan) chosen between the split caused by the path for ``R''. The shortest enclosing arc with the largest weight is fixed (``L'', green). \newline \textbf{Step~3:} After fixing the paths for ``R'' and ``L'', the path for ``D'' may only contain vertices which are located in the area enclosed by these paths. Accordingly, we calculate the new path for ``D'' by removing all other vertices (gray) and calculating a shortest path between the endpoints (magenta) in the remaining graph. } 

\end{figure}

\begin{itemize}
\item Notation: $T$ denotes the list of indices of the contour parts that are created by the decomposition and $i$ denotes the index for which a shortest path is currently calculated. We recall that $Q$ denotes the
parts of the contour that have not yet been fixed in the course of Algorithm~\ref{algo:optPath}.\\*[-2mm]
	\item Case 1: $T \cap Q = T$. \\*[2mm]
	As no path belonging to $T$ has been fixed, there are no constraints yet to be observed and we can simply calculate the shortest path for $i$.\\*[-2mm]

	\item Case 2: either $\exists l \in T \backslash Q$ with $\Gamma_l$ left of $\Gamma_i$ or $\exists r \in T \backslash Q$ with $\Gamma_r$ right of $\Gamma_i$.\\*[2mm]  
	Without loss of generality, we assume that we have to construct a shortest path $\Gamma_i$ subject to the constraint that it is to the right of an already fixed path $P_l$. Now, let $c$ be a point that is located between the left and right side of the split in the graph $g_i$ caused by $P_l$. Then, the order constraint is identical to finding the shortest path $p_i$ for which the circle formed by joining $p$ and $P_l$ encloses $c$. 
Lemma~\ref{lemma:shortestEnclosingArc}, stated at the end of this section, will show that $p_i$ is actually given by
	\begin{align*}	
		\Pi &= \{\text{sp}(g_i, v_i^l, u) \join uw \join \text{sp}(g_i, w, v_i^r):\; uw \text{ edge in $g_i$} \},\\*[2mm]
		p_i &= \argmin \{ d(p) : p \in \Pi;\; \ind(p \join P_l,c)= \pm 1 \}.
	\end{align*}
	Hence, $p_i$ can be constructed by a minor modification of the polynomial algorithm \cite{provan89} for shortest enclosing circles 
    in embedded graphs.\\*[-2mm]
	
	\item Case 3: $\exists l \in T \backslash Q$ with $\Gamma_l$ left of $\Gamma_i$ and $\exists r \in T \backslash Q$ with $\Gamma_r$ right of $\Gamma_i$.\\*[2mm]
	The shortest path for $\Gamma_i$, subject to the constraint that it is right of $\Gamma_l$ and left of $\Gamma_r$, can only contain vertices inside the circle formed by joining $P_l$ and $P_r$. Therefore, we construct $p_i$ as 
	the shortest path in a smaller graph, in which all vertices outside of this circle have been removed.
	
\end{itemize}

\begin{lemma}\label{lemma:shortestEnclosingArc}
Let $q = (q_1,\dots,q_n)$ be a path in a weighted planar graph $g = (V,E)$, let $c$ be a point considered\footnote{We can chose any point on $q$ for $c$ and treat it as if it were right of $q_-$ and left of $q_+$.} to be in the split along $q$ in $g[q]$ and define
\[
	\Pi = \{\spa(g[q], q_1, u) \join uv \join \spa(g[q], v, q_n):\; uv \in E \}.
	\]
Then the shortest walk $p_\star$ in $g$, subject to the constraint $\ind(p_\star \join \overleftarrow{q_\mp}, c) = \pm 1$, satisfies
\begin{equation}
	p_\star = \argmin\{d(p) :\; p \in \Pi;\; \ind(p \join \overleftarrow{q_\mp},c)= \pm 1 \}.
	\label{eq:shortestEnlcosingArc}
\end{equation}

\end{lemma}

\begin{proof}
We restrict ourselves to the case $\ind(p \join \overleftarrow{q})= 1$, because the proof for the other case differs just in the sign of some winding numbers. We will assume to the contrary that $p_\star$ is not given by \eqref{eq:shortestEnlcosingArc} and will get a contradiction. 

If there is more than one shortest walk $p_\star$, we choose the one which encloses the least number of vertices. If $p_\star \notin \Pi$, then there has to be a vertex $v \in p_\star$ with $p_\star = l \join v \join r$ such that $s_l = \text{sp}(g[q],q_1,v)$ and $s_r = \text{sp}(g[q],v,q_n)$ satisfy the following conditions
\begin{equation}\label{eq:proof:assumption}
		\ind( s_l \join r \join \overleftarrow{q}, c ) \ne 1,\qquad
		\ind( l \join s_r \join \overleftarrow{q}, c ) \ne 1.
\end{equation}
We will now show that there is no such vertex $v$. 

\medskip

\emph{Step 1.} To begin with, we prove for $W = p_\star \join \overleftarrow{q}$ that
\begin{equation}\label{eq:proof:step1}
    s_l \cap \interior(W) = \emptyset, \qquad
    s_r \cap \interior(W) = \emptyset.
\end{equation}
If \eqref{eq:proof:step1} would not hold then either $s_l$ or $s_r$ contains a path $p = (p_1,\dots,p_m)$ with $p_2,\dots,p_{m-1} \in \interior(W)$ and $p_1, p_m \in p_\star$. As $s_l$ and $s_r$ are shortest paths, such a subpath $p$ has to be the shortest path from $p_1$ to $p_m$. Consequently, the walk 
\[
	W' = p_\star[q_1,p_1] \join p \join p_\star[p_m,q_n] \join \overleftarrow{q}
\]
would satisfy
\[
	d(W') \le d(W), \qquad \ind(W',c) = 1, \qquad |\interior(W')| < |\interior(W)|,
\]
which contradicts our choice of $p_\star$. Therefore, \eqref{eq:proof:step1} holds.
\medskip

\emph{Step 2.} We now prove that 
\begin{equation}\label{eq:proof:step2}
	\begin{aligned}
		\ind(W^l_1,c) &= 1 \quad \text{with } W^l_1 = l \join \overleftarrow{s_l},\\*[2mm]
		\ind(W^r_1,c) &= 1 \quad \text{with } W^r_1 = r \join \overleftarrow{s_r}.
	\end{aligned}
\end{equation}
To this end, we consider the walk  
\begin{equation*}
	W^l = l \join \overleftarrow{s_l} \join s_l \join r \join \overleftarrow{q}
\end{equation*}
which consists of the two circles
\[
	W^l_1 = l \join \overleftarrow{s_l}, \qquad
	W^l_2 = s_l \join r \join \overleftarrow{q},
\]
and satisfies
\begin{equation}\label{eq:proof:enclose}
	\ind(W,c) = \ind(W^l,c) = \ind(W^l_1,c') + \ind(W^l_2,c') = 1.
\end{equation}
If $s_l \cap W = \emptyset$, then neither $W^l_1$ nor $W^l_2$ contains any vertex more than once. As $g$ is a planar graph, these walks correspond to simple closed curves in the complex plane and therefore their winding numbers around $c$ can only be $-1$, $0$ or $1$. If we also take \eqref{eq:proof:enclose} and \eqref{eq:proof:assumption} into account, the only possible option is 
\[
	\ind(W^l_1,c) = 1, \qquad
	\ind(W^l_2,c) = 0.
\]
Unfortunately, $s_l \cap W = \emptyset$ does not necessarily have to be satisfied. But $s_l$ cannot cross $l$ and $r \join \overleftarrow{q}$ because of \eqref{eq:proof:step1}, which means that $$s_l = \spa(g[q],q_1,v) = \spa(g[q,\overleftarrow{l},q\join\overleftarrow{r}],q_1,v).$$ Hence, we can find the following walks in $g[q,\overleftarrow{l},q\join\overleftarrow{r}]$
\begin{equation*}
		U^l_1 = l_+ \join \overleftarrow{s_l}, \qquad
		U^l_2 = s_l \join r_+ \join \overleftarrow{q_+},
\end{equation*}
which are equivalent to $W^l_1$ and $W^l_2$ but do not contain duplicate vertices. If we move the paths along the splits $\overleftarrow{l}$ and $q\join\overleftarrow{r}$ a little bit apart, then $U^l_1$ and $U^l_2$ correspond to simple connected curves, too. As we can move the split paths apart without crossing $c$ or any other part of $U^l_1$ or $U^l_2$, these walks can only have winding numbers of $-1$, $0$ or $1$ with respect to $c$. This means that $W^l_1$ and $W^l_2$ can only have these winding numbers even if $s_l \cap W \ne \emptyset$. This proves the first relation in \eqref{eq:proof:step2}. The second one follows likewise.

\medskip

\emph{Step 3.} We claim that
\begin{equation}\label{eq:proof:step3}
	\begin{aligned}
		q_n &\in \interior(W^l_1),\quad q_n \notin \interior(W^r_1),\\*[2mm]
		q_1 &\in \interior(W^r_1),\quad q_1 \notin \interior(W^l_1).
	\end{aligned}
\end{equation}
Combining \eqref{eq:proof:step1} and \eqref{eq:proof:step2} yields $\interior(W) \subseteq \interior(W^l_1)$. Therefore all vertices in $W$ either have to be in $\interior(W^l_1)$ or in $W^l_1$. It follows that we just have to show that $q_n \notin W^l_1$. We consider the following closed walk 
\[
	W' = W^l_1 \join W^r_1 = l \join r \join \overleftarrow{s_r} \join \overleftarrow{s_l}
\]
with
\[
	\ind(W') = \ind(W^l_1) + \ind(W^r_1) = 2.
\]
If $q_n \in W^l_1$, then $s_l[q_n,v] = \overleftarrow{s_r}$ and, consequently, $W'$ contains a circle that does not enclose any vertex. Removing this circle from $W'$ results in the walk
\[
	W'' = l \join r \join \overleftarrow{s_l}[q_n,q_1] 
\]
without changing the winding number. So $\ind(W'') = 2$, and the argument that we have used before to show $\ind(W^l_1,c) \in \{-1,0,1\}$ works for $W''$, too. Consequently, we get $q_n \notin W^l_1$. The second claim in \eqref{eq:proof:step3} follows likewise.

\medskip

\emph{Step 4.} (see Fig.~\ref{fig:proofStep4}) We combine the results of the previous steps to show that our initial assumption leads to a contradiction. As $c \in \interior(W^l_1) \cap \interior(W^r_1)$ by \eqref{eq:proof:step2}, the two circles $W^l_1$ and $W^r_1$ have a non empty intersection. Furthermore, because of \eqref{eq:proof:step3}, none of them is completely contained within the other. It follows that $W^l_1$ and $W^r_1$ have to cross each other at two or more points. One of these can be $v$, but there is actually no other vertex at which the circles could cross: $s_l$ cannot cross $r$ and, vice versa, $s_r$ cannot cross $l$ due to \eqref{eq:proof:step1}; also $s_l$ and $s_r$ cannot cross because they are both shortest paths. 
\end{proof}

\usetikzlibrary{shapes.geometric}
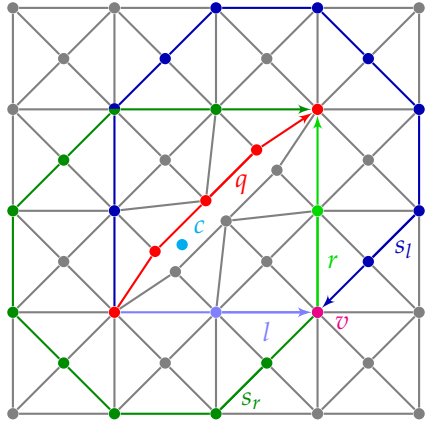
\begin{figure}[tbp]
	\begin{center}
		\begin{tikzpicture}
[vertex/.style={circle,minimum size=1.5mm,inner sep=0pt,fill=gray},
edge/.style={gray,thick},
splitVertex/.style={},
splitLabel/.style={red},
highlightLabel/.style={blue},
defaultLabel/.style={black},
newEdge/.style={},
changedEdge/.style={},
end/.style={->,>=latex'},
leftEnd/.style={<-,>=latex'},
rightEnd/.style={->,>=latex'},
highlightEdge/.style={blue},
highlightVertex/.style={blue},
sl/.style={blue!70!black},
pl/.style={blue!50!white},
sr/.style={green!55!black},
pr/.style={green!85!black}
,scale=0.9]\node[vertex,] (1) at (1. , 1.) {}; 
\node[vertex,sr] (2) at (1. , 2.5) {}; 
\node[vertex,sr] (3) at (1. , 4.) {}; 
\node[vertex,] (4) at (1. , 5.5) {}; 
\node[vertex,] (5) at (1. , 7.) {}; 
\node[vertex,sr] (6) at (2.5 , 1.) {}; 
\node[vertex,red] (7) at (2.5 , 2.5) {}; 
\node[vertex,sl] (8) at (2.5 , 4.) {}; 
\node[vertex,sr] (9) at (2.5 , 5.5) {}; 
\node[vertex,] (10) at (2.5 , 7.) {}; 
\node[vertex,sr] (11) at (4. , 1.) {}; 
\node[vertex,pl] (12) at (4. , 2.5) {}; 
\node[vertex,red] (13) at (3.85 , 4.15) {}; 
\node[vertex,sr] (14) at (4. , 5.5) {}; 
\node[vertex,sl] (15) at (4. , 7.) {}; 
\node[vertex,] (16) at (5.5 , 1.) {}; 
\node[vertex,pr] (17) at (5.5 , 2.5) {}; 
\node[vertex,pr] (18) at (5.5 , 4.) {}; 
\node[vertex,red] (19) at (5.5 , 5.5) {}; 
\node[vertex,sl] (20) at (5.5 , 7.) {}; 
\node[vertex,] (21) at (7. , 1.) {}; 
\node[vertex,] (22) at (7. , 2.5) {}; 
\node[vertex,sl] (23) at (7. , 4.) {}; 
\node[vertex,sl] (24) at (7. , 5.5) {}; 
\node[vertex,] (25) at (7. , 7.) {}; 
\node[vertex,sr] (26) at (1.75 , 1.75) {}; 
\node[vertex,] (27) at (1.75 , 3.25) {}; 
\node[vertex,sr] (28) at (1.75 , 4.75) {}; 
\node[vertex,] (29) at (1.75 , 6.25) {}; 
\node[vertex,] (30) at (3.25 , 1.75) {}; 
\node[vertex,red] (31) at (3.1 , 3.4) {}; 
\node[vertex,] (32) at (3.25 , 4.75) {}; 
\node[vertex,sl] (33) at (3.25 , 6.25) {}; 
\node[vertex,sr] (34) at (4.75 , 1.75) {}; 
\node[vertex,] (35) at (4.75 , 3.25) {}; 
\node[vertex,red] (36) at (4.6 , 4.9) {}; 
\node[vertex,] (37) at (4.75 , 6.25) {}; 
\node[vertex,] (38) at (6.25 , 1.75) {}; 
\node[vertex,sl] (39) at (6.25 , 3.25) {}; 
\node[vertex,] (40) at (6.25 , 4.75) {}; 
\node[vertex,sl] (41) at (6.25 , 6.25) {}; 
\node[vertex,splitVertex,] (42) at (3.4 , 3.1) {}; 
\node[vertex,splitVertex,] (43) at (4.15 , 3.85) {}; 
\node[vertex,splitVertex,] (44) at (4.9 , 4.6) {}; 
\draw[edge,](1) -- (2) ; 
\draw[edge,](1) -- (6) ; 
\draw[edge,](1) -- (26) ; 
\draw[edge,sr](2) -- (3) ; 
\draw[edge,](2) -- (7) ; 
\draw[edge,sr](2) -- (26) ; 
\draw[edge,](2) -- (27) ; 
\draw[edge,](3) -- (4) ; 
\draw[edge,](3) -- (8) ; 
\draw[edge,](3) -- (27) ; 
\draw[edge,sr](3) -- (28) ; 
\draw[edge,](4) -- (5) ; 
\draw[edge,](4) -- (9) ; 
\draw[edge,](4) -- (28) ; 
\draw[edge,](4) -- (29) ; 
\draw[edge,](5) -- (10) ; 
\draw[edge,](5) -- (29) ; 
\draw[edge,](6) -- (7) ; 
\draw[edge,sr](6) -- (11) ; 
\draw[edge,sr](6) -- (26) ; 
\draw[edge,](6) -- (30) ; 
\draw[edge,sl](7) -- (8) ; 
\draw[edge,pl](7) -- (12) ; 
\draw[edge,](7) -- (26) ; 
\draw[edge,](7) -- (27) ; 
\draw[edge,](7) -- (30) ; 
\draw[edge,red](7) -- (31) ; 
\draw[edge,changedEdge,](7) -- (42) ; 
\draw[edge,sl](8) -- (9) ; 
\draw[edge,](8) -- (13) ; 
\draw[edge,](8) -- (27) ; 
\draw[edge,](8) -- (28) ; 
\draw[edge,](8) -- (31) ; 
\draw[edge,](8) -- (32) ; 
\draw[edge,](9) -- (10) ; 
\draw[edge,sr](9) -- (14) ; 
\draw[edge,sr](9) -- (28) ; 
\draw[edge,](9) -- (29) ; 
\draw[edge,](9) -- (32) ; 
\draw[edge,sl](9) -- (33) ; 
\draw[edge,](10) -- (15) ; 
\draw[edge,](10) -- (29) ; 
\draw[edge,](10) -- (33) ; 
\draw[edge,](11) -- (12) ; 
\draw[edge,](11) -- (16) ; 
\draw[edge,](11) -- (30) ; 
\draw[edge,sr](11) -- (34) ; 
\draw[edge,rightEnd,pl](12) -- (17) ; 
\draw[edge,](12) -- (30) ; 
\draw[edge,](12) -- (34) ; 
\draw[edge,](12) -- (35) ; 
\draw[edge,changedEdge,](12) -- (42) ; 
\draw[edge,changedEdge,](12) -- (43) ; 
\draw[edge,](13) -- (14) ; 
\draw[edge,red](13) -- (31) ; 
\draw[edge,](13) -- (32) ; 
\draw[edge,red](13) -- (36) ; 
\draw[edge,](14) -- (15) ; 
\draw[edge,rightEnd,sr](14) -- (19) ; 
\draw[edge,](14) -- (32) ; 
\draw[edge,](14) -- (33) ; 
\draw[edge,](14) -- (36) ; 
\draw[edge,](14) -- (37) ; 
\draw[edge,sl](15) -- (20) ; 
\draw[edge,sl](15) -- (33) ; 
\draw[edge,](15) -- (37) ; 
\draw[edge,](16) -- (17) ; 
\draw[edge,](16) -- (21) ; 
\draw[edge,](16) -- (34) ; 
\draw[edge,](16) -- (38) ; 
\draw[edge,pr](17) -- (18) ; 
\draw[edge,](17) -- (22) ; 
\draw[edge,sr](17) -- (34) ; 
\draw[edge,](17) -- (35) ; 
\draw[edge,](17) -- (38) ; 
\draw[edge,leftEnd,sl](17) -- (39) ; 
\draw[edge,rightEnd,pr](18) -- (19) ; 
\draw[edge,](18) -- (23) ; 
\draw[edge,](18) -- (35) ; 
\draw[edge,](18) -- (39) ; 
\draw[edge,](18) -- (40) ; 
\draw[edge,changedEdge,](18) -- (43) ; 
\draw[edge,changedEdge,](18) -- (44) ; 
\draw[edge,](19) -- (20) ; 
\draw[edge,](19) -- (24) ; 
\draw[edge,leftEnd,red](19) -- (36) ; 
\draw[edge,](19) -- (37) ; 
\draw[edge,](19) -- (40) ; 
\draw[edge,](19) -- (41) ; 
\draw[edge,changedEdge,](19) -- (44) ; 
\draw[edge,](20) -- (25) ; 
\draw[edge,](20) -- (37) ; 
\draw[edge,sl](20) -- (41) ; 
\draw[edge,](21) -- (22) ; 
\draw[edge,](21) -- (38) ; 
\draw[edge,](22) -- (23) ; 
\draw[edge,](22) -- (38) ; 
\draw[edge,](22) -- (39) ; 
\draw[edge,sl](23) -- (24) ; 
\draw[edge,sl](23) -- (39) ; 
\draw[edge,](23) -- (40) ; 
\draw[edge,](24) -- (25) ; 
\draw[edge,](24) -- (40) ; 
\draw[edge,sl](24) -- (41) ; 
\draw[edge,](25) -- (41) ; 
\draw[edge,changedEdge,](35) -- (43) ; 
\draw[edge,newEdge,](42) -- (43) ; 
\draw[edge,newEdge,](43) -- (44) ; 
\node[vertex,cyan,label={above right,label distance=-1,cyan:$c$}] (*) at (3.5,3.5) {};%node for c
\draw[edge,red] (13) -- (36) node[near end,below]{$q$}; %label for path q
\draw[edge,pl,rightEnd] (12) -- (17) node[midway,below]{$l$}; %label for path l
\draw[edge,pr] (17) -- (18) node[midway,right]{$r$}; %label for path r
\draw[edge,sl] (23) -- (39) node[near start,below,inner ysep=5]{$s_l$}; %label for path sl
\draw[edge,sr] (34) -- (11) node[near start,below,inner ysep=5]{$s_r$};%label for path sr
\node[vertex,sl,shape=semicircle,minimum size=0.75mm,anchor=south] (s) at (2.5, 5.5) {};%top half of the vertex a which s_l and s_r cross each other is colored in the color of s_l
\node[vertex,magenta,label={below right,label distance=2,yshift=5,magenta:$v$}] at (17) {};%the vertex v is colored in magenta because we cannot really assign it to any of the paths ending or starting at v

\end{tikzpicture}
	\end{center}
	\caption{Illustration of Step 4 in the proof of Lemma~\ref{lemma:shortestEnclosingArc}.}
	\label{fig:proofStep4}
\end{figure}

\section{Implementation Details}\label{sec:implDetails}

\subsection{The Weights}
The weight $d_e$ of an edge $e$ should be an approximation of
\[
\int_e \| G(z) - \I \| \,d |z|
\] 
with a suitable matrix norm. Our experiments indicate that we generally need fewer collocation points
 if we aim at minimizing all components of $G-I$ instead of just focussing on its largest component, for which reason
 we choose the Frobenius (or Hilbert-Schmidt) norm. The integral is sufficiently well approximated by the two point trapezoidal quadrature rule (we recall that the aim of optimizing the weight is just preconditioning, that is, getting a particular good \emph{order
 of magnitude} of the condition number). We thus take
\[
	d_e = \frac12 |b-a| ( \|G(b) - \I \|_F + \|G(a) - \I \|_F )
\]
as the weight of an edge $e$ with the endpoints $a$ and $b$.

\subsection{The Graph}
The algorithm of §\ref{sec:algorithm} is based on {\em planar} graphs. If the graph were not planar, paths
could cross each other even without having any vertices in common and, therefore, the graph splitting described in §\ref{sec:deform} would not ensure that paths calculated by Algorithm \ref{algo:optPath} do not cross. 
We choose planar graphs built from rectangular grids to which a vertex in the center of each box is added that is connected to the vertices of the that box. 
Such a graph is chosen to subdivide a rectangle that contains all finite endpoints of $\Gamma$. We take 
this rectangle large enough so that outside of it $\|G - \I\|_F$ is below machine precision on all arcs with an infinite endpoint, see Fig.~\ref{fig:rectangle}. For numerical purposes, the jump matrix $G$ is then  indistinguishable from the identity matrix in the
exterior of this rectangle: the RHP needs only to be solved in the interior.

\begin{figure}[tbp]
	\begin{center}
	\begin{minipage}{0.47\textwidth}
		\begin{center}
		\includegraphics[width=\textwidth]{jumpScalePII--10-black}
		\end{center}
	\end{minipage}
	\hfil
%	\begin{minipage}{0.32\textwidth}
%		\begin{center}
%		\includegraphics[width=\textwidth]{jumpScalePII--10}
%		\end{center}
%	\end{minipage}
%	\hfil
	\begin{minipage}{0.47\textwidth}
		\begin{center}
		\includegraphics[width=\textwidth]{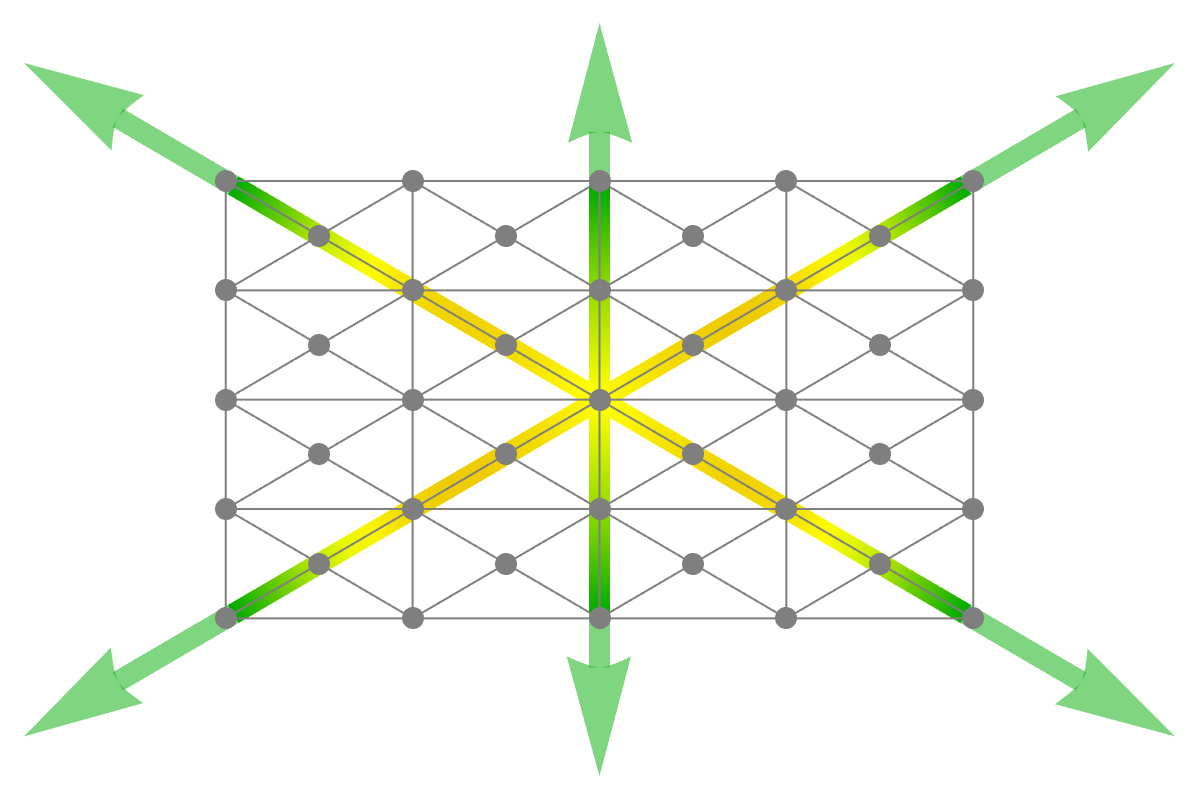}
		\end{center}
	\end{minipage}\\
	\end{center}
	\vspace{0.25cm}
	\caption{The rectangle to be covered by the grid is determined by the condition $\|G-\I\|_F > 10^{-16}$ along $\Gamma$. 
		The color coding is as in Fig.~\ref{fig:results}. }\label{fig:rectangle}
\end{figure}

\subsection{Contour Simplification}\label{subsec:contourSimplification}
The algorithm described in §\ref{sec:algorithm} returns a contour composed of a set of paths in the underlying graph. 
The collocation method of Olver \cite{Olver:2011:NSR:1967343.1967345}, which is 
finally employed for the numerical solution of the RHP, would have to place individual Chebyshev points on each smooth (that is, linear) part of this
piecewise linear contour. For efficiency reasons it would thus be preferable to have a contour with fewer breakpoints. Consequently, for each optimized path, we calculate a coarse piecewise linear approximation that has about the same weight. Quite often just a straight line connecting the endpoints of a path is already sufficient approximation. Fig.~\ref{fig:ContourSimplification} shows an example of this simplification process when applied to the
final contour of Fig.~\ref{fig:results}: it cuts the number of collocation points by more than a factor of two
while keeping the order of magnitude of the condition number constant.
 
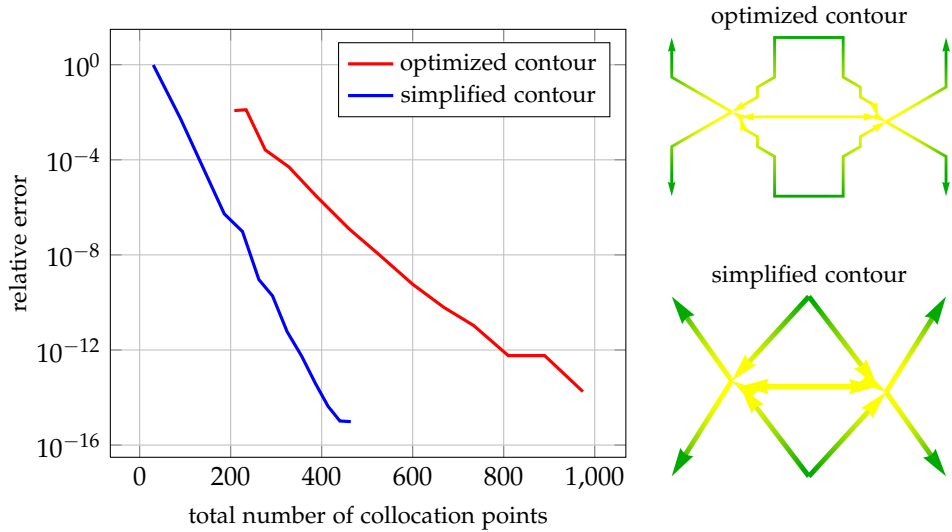
\begin{figure}
	\begin{minipage}{0.65\textwidth}
	    ~~\\*[2mm]
		\begin{tikzpicture}
			\begin{semilogyaxis}[xlabel={\small total number of collocation points},
				ylabel={\small relative error},
				grid=major,
				no markers,
				legend pos=north east,
				]
								\addplot[very thick,red] file {plots/pII-10Conv-Adaptive.dat};
								\addplot[very thick,blue] file {plots/pII-10Conv-Adaptive-Simplified.dat};
				\legend{{\small optimized contour},{\small simplified contour}}
			\end{semilogyaxis}
		\end{tikzpicture}
	\end{minipage}
	\hfill
	\begin{minipage}{0.3\textwidth}
		\begin{center}
			{\small optimized contour}
			\includegraphics[width=\textwidth]{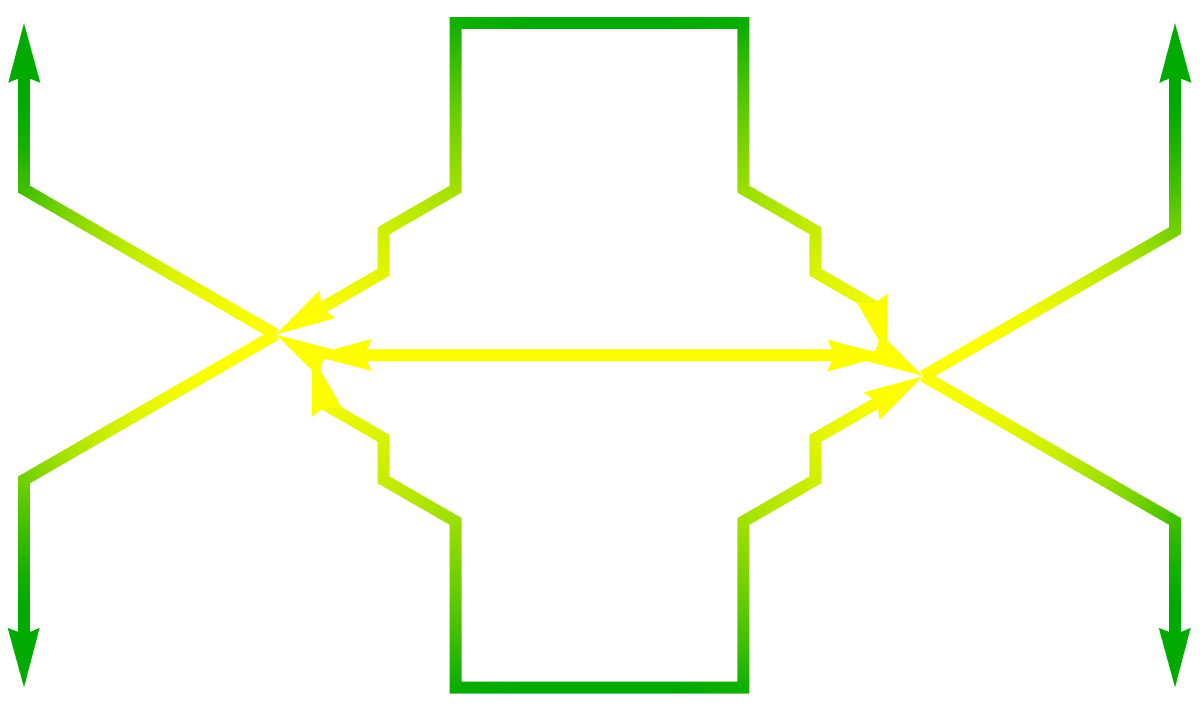}
			\vspace{0.2cm}\\
			{\small simplified contour}
			\includegraphics[width=\textwidth]{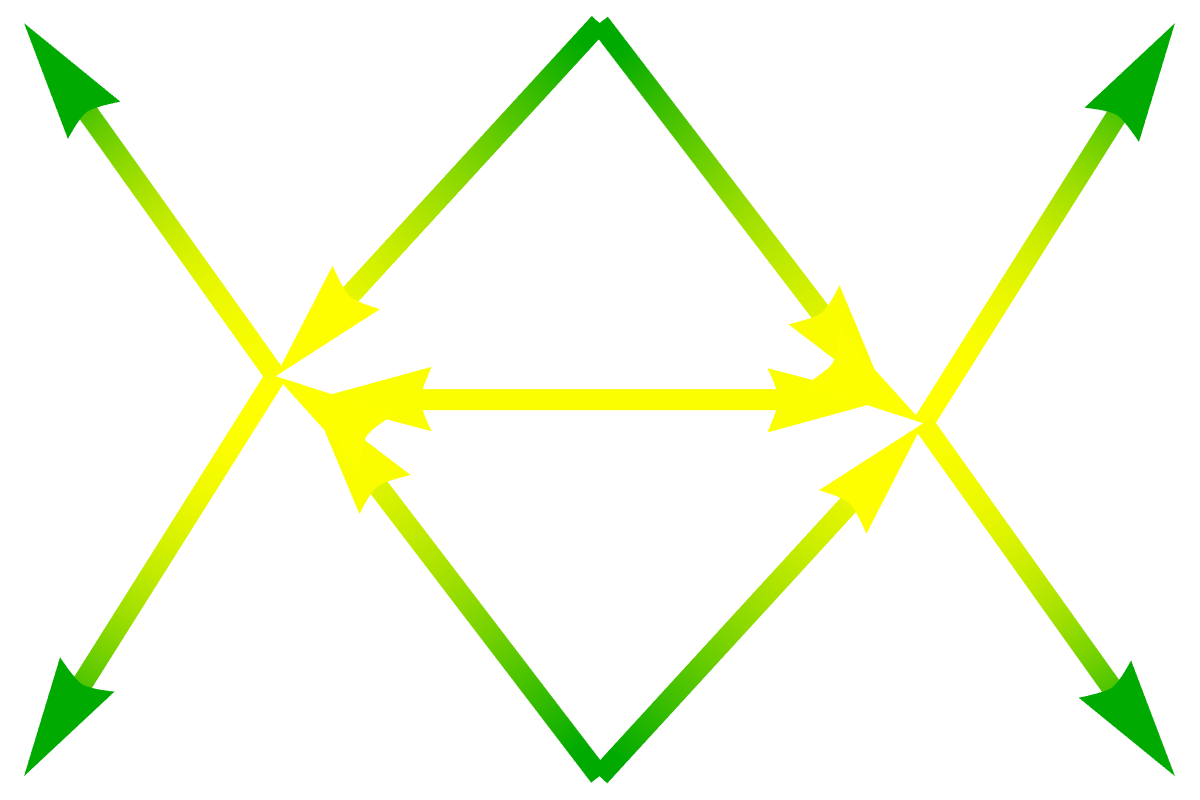}
			\vspace{0.4cm}\\
		\end{center}
	\end{minipage}
	\caption{Improvement of the convergence rate by contour simplification: 
	the simplified contour needs only about half the number of collocation points to reach the same accuracy as 
	the optimized contour of Fig.~\ref{fig:results}.d (the color coding is the same as there). This
	simplification does not, however, worsen the order of magnitude of the condition number which grows from about 140 to just about 200. The similarity with the manually constructed contour in Fig.~\ref{fig:olver} is even more
	striking after this simplification step.}\label{fig:ContourSimplification} 
\end{figure}

\section*{Conclusion}

The numerical results of this paper
show that our algorithm can significantly reduce the condition number of RHPs. As a feature,
this algorithm does not require any input, or knowledge, from the user other than the RHP at hand. Besides
being thus a very convenient tool for the numerical solution of RHPs, the deformations automatically constructed by this algorithm might even turn out to be useful for determining first drafts of
suitable deformations in the analytic study of RHPs.

\subsection*{Acknowledgement} 
This research was supported by the DFG-Collaborative Research Center, TRR 109, ``Discretization in Geometry and Dynamics''.

\bibliographystyle{plain}
\bibliography{autodeform}

\end{document}